\documentclass[english,notitlepage]{amsart}
\usepackage[latin9]{inputenc}
\usepackage{geometry}
\usepackage{amsthm}
\usepackage{amsmath, amscd}
\usepackage{amssymb}
\usepackage{esint}
\usepackage{xcolor}

\makeatletter
\numberwithin{equation}{section}
\numberwithin{figure}{section}

\theoremstyle{plain}
\newtheorem{thm}{Theorem}
  \theoremstyle{plain}
  \numberwithin{thm}{section}
  \newtheorem{cor}[thm]{Corollary}
  \theoremstyle{plain}
  \newtheorem{lem}[thm]{Lemma}
  \theoremstyle{plain}
   \newtheorem{proposition}[thm]{Proposition}
    \newtheorem{definition}[thm]{Definition}
  \theoremstyle{remark}
  \newtheorem{rem}[thm]{Remark}
  \newtheorem{ex}[thm]{Example}
 
  \def\Ddots{\mathinner{\mkern1mu\raise\p@
\vbox{\kern7\p@\hbox{.}}\mkern2mu
\raise4\p@\hbox{.}\mkern2mu\raise7\p@\hbox{.}\mkern1mu}}
\makeatother

\usepackage{babel}

\newcommand{\eps}{\varepsilon}

\newcommand{\norm}[1]{\left\| #1 \right\|}
\newcommand{\mklm}[1]{\left\{ #1 \right\}}
\newcommand{\eklm}[1]{\left\langle #1 \right\rangle}

\renewcommand{\d}{\,d}
\newcommand{\N}{{\mathbb N}}
\newcommand{\Z}{{\mathbb Z}}
\newcommand{\C}{{\mathbb C}}
\newcommand{\Ccal}{{\mathcal C}}
\newcommand{\R}{{\mathbb R}}

\newcommand{\F}{{\mathcal F}}

\newcommand{\M}{{\mathcal M}}
\newcommand{\T}{{\mathcal T}}

\newcommand{\Q}{{\mathbb Q}}
\renewcommand{\O}{{\mathcal O}}

\newcommand{\1}{{\bf 1}}
\renewcommand{\epsilon}{\varepsilon}

\renewcommand{\rho}{\varrho}

\newcommand{\Cinft}{{\rm C^{\infty}}}
\newcommand{\CT}{{\rm C^{\infty}_c}}

\renewcommand{\L}{{\rm L}}
\newcommand{\Lcal}{{\mathcal L}}

\renewcommand{\S}{{\mathcal S}}

\newcommand{\GL}{\mathrm{GL}}
\newcommand{\PGL}{\mathrm{PGL}}
\newcommand{\SL}{\mathrm{SL}}
\newcommand{\SO}{\mathrm{SO}}
\newcommand{\SU}{\mathrm{SU}}

\newcommand{\g}{{\bf \mathfrak g}}
\renewcommand{\k}{{\bf \mathfrak k}}
\renewcommand{\t}{{\bf \mathfrak t}}
\newcommand{\p}{{\bf \mathfrak p}}

\newcommand{\Ad}{\mathrm{Ad}\,}
\newcommand{\ad}{\mathrm{ad}\,}
\newcommand{\sgn}{\mathrm{sgn}\,}
\newcommand{\id}{\mathrm{id}\,}

\renewcommand{\det}{\mathrm{det}\,}
\renewcommand{\Im}{\mathrm{Im}\,}

\newcommand{\dist}{\mathrm{dist}\,}

\newcommand{\Crit}{\mathrm{Crit}}

\DeclareMathOperator{\supp}{supp\,}

\DeclareMathOperator{\tr}{tr}
\DeclareMathOperator{\gd}{\partial}

\DeclareMathOperator{\grad}{grad}

\newcommand{\dbar}{{\,\raisebox{-.1ex}{\={}}\!\!\!\!d}}

\newcommand{\bdm}{\begin{displaymath}}
\newcommand{\edm}{\end{displaymath}}
\newcommand{\bq}{\begin{equation}}
\newcommand{\eq}{\end{equation}}
\newcommand{\bqn}{\begin{equation*}}
\newcommand{\eqn}{\end{equation*}}

\newcommand{\bsl}{\backslash}
\newcommand{\cR}{\mathcal R}

\newcommand{\bH}{\mathbb H}
\newcommand{\bA}{\mathbb A}

\newcommand{\dista}{\mathrm{dist}_1}

\def\bi{{\mathbf{i}}}
\def\bj{{\mathbf{j}}}
\def\bk{{\mathbf{k}}}

\def\uG{{\underline{G}}}

\begin{document}
\author{Pablo Ramacher and Satoshi Wakatsuki}
\title[Subconvex bounds for Hecke--Maass forms on compact arithmetic quotients]{Subconvex bounds for Hecke--Maass forms on compact arithmetic quotients of semisimple  Lie groups}
\address{Fachbereich 12 Mathematik und Informatik, Philipps-Universit\"at Marburg, Hans--Meerwein-Str., 35032 Marburg, Germany}
\email{ramacher@mathematik.uni-marburg.de}
\address{Faculty of Mathematics and Physics, Institute of Science and Engineering, Kanazawa University, Kakumamachi, Kanazawa, Ishikawa, 920-1192, Japan}
\email{wakatsuk@staff.kanazawa-u.ac.jp}

\date{September 8, 2018}

\begin{abstract}
Let $H$ be a semisimple algebraic group, $K$ a maximal compact subgroup of $G:=H(\R)$, and $\Gamma\subset H(\Q)$ a congruence arithmetic subgroup.
In this paper, we generalize existing subconvex bounds for Hecke--Maass forms on the  locally symmetric space $\Gamma \bsl G/K$  to corresponding bounds on the arithmetic quotient $\Gamma \bsl G$ for cocompact lattices using the spectral function of an elliptic operator. The bounds obtained  extend known subconvex bounds for automorphic forms to non-trivial $K$-types, yielding  subconvex bounds for new classes of automorphic representations, and constitute  subconvex bounds for eigenfunctions on compact manifolds with both positive and negative sectional curvature. We also obtain new subconvex bounds for holomorphic modular forms in the weight aspect.
\end{abstract}

\maketitle

\setcounter{tocdepth}{1}
\tableofcontents{}

\section{Introduction}
\label{sec:1}

Let $M$ be  a closed\footnote{ By a closed manifold we shall understand a compact boundaryless manifold.}  Riemannian manifold $M$ of dimension $d$ and $P_0:\Cinft(M) \rightarrow \L^2(M)$ an elliptic classical pseudodifferential operator on $M$ of degree $m$, where $\Cinft(M)$ denotes the space of smooth functions on $M$ and $\L^2(M)$ the space of square-integrable functions on $M$.  Assume that $P_0$  is positive and symmetric. Denote its  unique self-adjoint extension by  $P$ with  the $m$-th Sobolev space as domain, and let $\mklm{\phi_j}_{j\geq 0}$ be an orthonormal basis of $\L^2(M)$ consisting of eigenfunctions of $P$ with eigenvalues $\mklm{\lambda_j}_{j \geq 0}$ repeated according to their multiplicity.  By a classical result of Avacumovic, Levitan, and H\"ormander \cite{avacumovic, levitan52, hoermander68} one has for any $j \in \N$ the \emph{convex bound}\footnote{Here and in what follows we shall write $a \ll_\gamma b$ for two real numbers $a$ and $b$ and a variable $\gamma$, if there exists a constant $C_\gamma>0$ depending only on $\gamma$ such that $|a| \leq C_\gamma b$. If there are no relevant variables involved, we shall simply write $a\ll b$.}
\bq
\label{eq:ALH}
\norm{\phi_j}_\infty \ll \lambda_j^{\frac{d-1}{2m}}.
\eq
If the  $\phi_j$ are eigenfunctions of a larger family of commuting differential operators on $M$ containing $P_0$, this bound can be improved. Thus, assume that $M$ carries an isometric action of a compact Lie group $K$ such that all orbits have the same dimension $\kappa\leq d-1$.  Denote by $\widehat K$  the set of  equivalence classes of irreducible unitary  representations of $K$, which can be identified with the set of irreducible characters of $K$. Suppose further that $P$ commutes with the family of differential operators generated by the action of $K$, so that the eigenfunctions $\phi_j$ can be chosen to be compatible with the  Peter--Weyl decomposition of $\L^2(M)$ into $\sigma$-isotypic components $\L^2_\sigma(M)$,  where $\sigma \in \widehat K$. It was then shown in \cite{ramacher16, ramacher18} that the \emph{equivariant convex bound}
\bq
\label{eq:R}
\norm{\phi_j}_\infty \ll  \Big (d_\sigma \sup_{u \leq \lfloor{\kappa}/2+1\rfloor} \norm{D^u \sigma}_\infty\Big )^{1/2} \lambda_j^{\frac{d-\kappa-1}{2m}}, \qquad \phi_j \in \L_\sigma^2(M),
\eq
 holds, where $d_\sigma$ denotes the dimension of a representation of class $\sigma$, and $D^u$ is a differential operator of order $u$ on $K$. If $K=T$ is a torus, one actually has the almost sharp estimate 
\bq
\label{eq:R2}
\norm{\phi_j}_\infty \ll   \lambda_j^{\frac{d-\kappa-1}{2m}}, \qquad \phi_j \in \L_\sigma^2(M), \qquad \sigma \in \mathcal{W}_{\lambda_j},
\eq
where $\mathcal{W}_\lambda$ denotes the  subset of $K$-types occuring in the Peter-Weyl decomposition of $\L^2(M)$ that grow at most with rate 
$ \lambda^{1/m}/\log \lambda.$
\footnote{The estimate is almost sharp in the sense that as a consequence of the equivariant Weyl law $T$-types in $\L^2(M)$ can grow at most with rate $\lambda_j^{1/m}$, see \cite{ramacher18}.}

The bounds \eqref{eq:ALH} and \eqref{eq:R} are  known to be sharp in the eigenvalue aspect on the standard $d$-sphere, but if the considered eigenfunctions are  joint eigenfunctions of an even larger family of commuting operators, they can be improved. Thus,  let $G$ be a semisimple real Lie group, $K$ a maximal compact subgroup of $G$, $\Gamma\subset G$ a lattice, and $Y:=\Gamma \bsl G / K$ the corresponding locally symmetric space of dimension $d$ and rank $r$. If  $\mklm{\psi_j}_{j\geq 0}$ constitutes an orthonormal basis in $\L^2(Y)$ of  simultaneous eigenfunctions  of the full ring of invariant differential operators on $Y$, which is isomorphic to a finitely generated polynomial ring in $r$ variables and contains the Beltrami--Laplace operator $\Delta$, Sarnak \cite{sarnak_letter} was able to show the \emph{spherical convex bound}
\bq
 \label{eq:S}
\norm{\psi_j|_{\Omega}}_\infty \ll_{\Omega} \lambda_j^{\frac{d-r}{4}}
\eq
 for arbitrary compacta $\Omega \subset Y$,  $\lambda_j$ being the Beltrami--Laplace eigenvalue of $\psi_j$. From an arithmetic point of view, there is still an additional family of commuting operators on $Y$ given by the Hecke operators, and in the case $G=\SL(2,\R)$ and $K=\SO(2)$, Iwaniec and Sarnak \cite{iwaniec-sarnak95} were able to strengthen the bound \eqref{eq:S} for certain compact locally symmetric spaces $Y=\Gamma \bsl \bH$ of rank $r=1$, given as quotients  of the complex upper half plane $\bH\simeq G/K$ by suitable congruence arithmetic lattices $\Gamma$, and proved for any $\eps>0$ and $j \in \N$ the substantially stronger \emph{spherical subconvex bound}
 \bq
 \label{eq:iwaniec-sarnak95} 
\norm{\psi_j}_\infty \ll_\eps \lambda_j^{\frac{5}{24}+\eps}, 
\eq
 provided that the $\psi_j$ are also eigenfunctions of the ring of Hecke operators on $\L^2(\Gamma \bsl \bH)$.  More generally, if $H$ is  a semisimple algebraic group over $\Q$ satisfying certain conditions, $\Gamma \subset H(\Q)$ an arithmetic congruence lattice, and $G=H(\R)$, Marshall \cite{Marshall2017} was able to strengthen the bound \eqref{eq:S}  and prove \emph{spherical subconvex bounds} of the form 
 \bq
 \label{eq:M}
\norm{\psi_j|_{\Omega}}_\infty \ll_\Omega \lambda_j^{\frac{d-r}{4}-\delta}
\eq
 for some $\delta >0$ and arbitrary compacta $\Omega \subset Y$, if the $\psi_j$ are also eigenfunctions of the ring of Hecke operators on $\L^2(Y)$, 
 generalizing previous work of Blomer-Maga \cite{BM1,BM2} and Blomer-Pohl \cite{BP}, among others. In fact, for negatively curved manifolds, much better bounds 
are expected to hold generically, the bound  
 \eqref{eq:iwaniec-sarnak95} being the strongest known bound up to now. The estimates \eqref{eq:S}--\eqref{eq:M} represent bounds for  automorphic forms on $G$ which are right $K$-invariant, and for this reason are   called \emph{spherical}. \\ 

In this paper, left $\Gamma$-invariant functions on $G$ which are simultaneous eigenfunctions of an invariant elliptic differential operator and some {module} of Hecke operators will be called \emph{Hecke--Maass forms of rank $1$}. This class encompasses the usual concept  of an automorphic form on  $G$, and coincides with it in the rank $1$ case, compare Section \ref{sec:autrepr} and \ref{sec:autreprII}. 
Nevertheless, note that a Hecke--Maass form of rank $1$ is not necessarily an eigenfunction of the full ring of invariant differential operators, since one can  choose a very small submodule of the ring of Hecke operators, see Remark \ref{rem:rank} for details. 
The goal of this paper is to extend the spherical  subconvex bounds \eqref{eq:iwaniec-sarnak95} and \eqref{eq:M} to  non-spherical situations, that is, to non-trivial $K$-types in the Peter-Weyl decomposition of $\L^2(\Gamma \bsl G)$ for a large class of compact arithmetic quotients $\Gamma \bsl G$, sharpening the bounds \eqref{eq:ALH} and \eqref{eq:R} in case that the eigenfunctions $\phi_j$ are Hecke--Maass forms.  \\

As our first main result, we extend the bound \eqref{eq:iwaniec-sarnak95} to automorphic forms on $G$ of arbitrary $K$-type and  Nebentypus character. Thus,  let $\cR$ be an Eichler order in an indefinite division quaternion algebra $A$ over $\Q$. Denote by $N(x)$  the reduced norm of an element $x\in A$, and write $\cR(m):=\mklm{\alpha \in \cR\mid  N(\alpha) =m}$ for any $m \in \N_\ast$. Choose an embedding $\theta:\sqcup_{m=1}^\infty \cR(m) \rightarrow G$, and set $\Gamma:=\theta(\cR(1))$. Then $\Gamma$ constitutes a congruence arithmetic subgroup, and $\Gamma\bsl \bH\simeq \Gamma \bsl G/ K$ becomes a compact hyperbolic surface. Now, let $\chi$ be a Nebentypus character on $\Gamma$, and denote by  $\L^2_\chi(\Gamma \bsl G)$ the Hilbert space of measurable functions on $G$ such that 
\bqn
f(\gamma x)=\chi(\gamma)\, f(x),  \qquad \gamma\in\Gamma, \,  x\in G, \qquad 
\norm{f}:=\bigg (\int_{\Gamma \bsl G} |f(x)|^2 dx\bigg ) ^{1/2} < \infty.
 \eqn
  The space $\L^2_\chi(\Gamma \bsl G)$ can be regarded as a closed subspace in $\L^2(\Gamma_\chi \bsl G)$, where $\Gamma_\chi:=\ker \chi$. Identifying $\cR(n)$ with its image $\theta(\cR(n))$ for each $n$ prime to a fixed natural number which depends only on $\cR$, the finite cosets $\Gamma\bsl \cR(n)$ give rise to  Hecke operators 
on $\L^2_\chi(\Gamma \bsl G)$.  Now, with the identification $K\simeq S^1\simeq [0,2\pi)$,  any  $K$-type $\sigma_l \in \widehat K$ can be realized as  the character $\sigma_l(\theta)=e^{il\theta}$, $\theta \in [0,2\pi)$, $l \in \Z$, and we denote by $\L^2_{\sigma_l,\chi}(\Gamma \bsl G)$ the $\sigma_l$-isotypic component  of  $\L^2_\chi(\Gamma \bsl G)$.  It is then shown in Theorem \ref{thm:24.11.2016}   that for any orthonormal basis  $\mklm{\phi_j}_{j\geq 0}$ of $\L^2(\Gamma_\chi \bsl G)$ consisting of Hecke--Maass forms (of rank $1$) with Beltrami--Laplace eigenvalues $0 \leq \lambda_0 \leq \lambda_1 \leq \lambda_2 \leq \cdots$ and compatible with the Peter-Weyl decomposition one has the \emph{hybrid subconvex bound}
 \bq 
\label{eq:11.1.17a}
\norm{\phi_j}_\infty \ll_\eps  \,  \lambda_j^{\frac{5}{24}+\eps}, \qquad \phi_j \in \L^2_{\chi}(\Gamma\bsl G),
\eq
for arbitrary small  $\eps>0$  in the eigenvalue and isotypic aspect.  This bound is the first sharpening the bound \eqref{eq:R2} for arbitrary $K$-types. If $\sigma_l$ and $\chi$ are  trivial, one recovers the spherical subconvex bound \eqref{eq:iwaniec-sarnak95}. 
Note that  \eqref{eq:11.1.17a} is a  subconvex bounds on a manifold which does have both positive and negative sectional curvature.  
It is stated from the perspective of elliptic operator theory, which is the natural one in our approach, while in the theory of automorphic forms it is more common to work within a representation-theoretic framework, and use the Casimir operator $\Ccal$ of $G$ instead of the Beltrami--Laplace operator $\Delta$, the former being no longer elliptic.  But since on $\L^2_{{\sigma_l},\chi}(\Gamma\bsl G)$ the operators in question are related according to
$ 
\Delta =  -\Ccal + \frac {l^2} 4 \id ,
$
the bound   \eqref{eq:11.1.17a}  can be rephrased accordingly.  Thus,  for any Hecke eigenform $\phi \in \L^2_{{\sigma_l},\chi}(\Gamma\bsl G)$ satisfying $\norm{\phi}_{\L^2}=1$ and  $\Ccal \phi = \frac{s^2-1}8\phi$ one has the \emph{hybrid subconvex bound} 
 \bqn 
\norm{\phi}_\infty \ll_\eps     ( 1-s^2+2l^2 )^{\frac{5}{24}+\eps}, 
\eqn
see Theorem \ref{thm:28.8.2017}.   In this way, we obtain  subconvex bounds for new classes of automorphic representations, in particular for the discrete series $D_s$ and their limits $D_{\pm,0}$, as well as the principal series $H(1,s)$, compare Section \ref{sec:autrepr}. Let us note that for fixed $s$ we obtain the bound  $\norm{\phi_j}_\infty \ll_\eps   (1+|l|)^{\frac{5}{12}+\eps}$ for any $ \phi_j \in \L^2_{{\sigma_l},\chi}(\Gamma\bsl G)$. This agrees with results of Venkatesh \cite[p. 993]{venkatesh10}, though by work of  Reznikov \cite[Theorem 1.5]{reznikov08} one has in this case the much better bound  $\norm{\phi_j}_\infty \ll_\eps   (1+|l|)^{\frac{1}{3}+\eps}$. Nevertheless, our results  do  imply new results for a classical automorphic form $f:\bH \rightarrow \C$  of weight $l\in \N$ and arbitrary Nebentypus character, for which we show in \eqref{eq:3.9.2018} the subconvex bound
\bqn
\|{ f}\|_\infty \ll_\eps l^{\frac{5}{12}+\eps}
\eqn
 in the weight aspect. The best previously known  subconvex bound, proved by Das and Sengupta \cite{DS},  had the exponent   $\frac{1}{2}-\frac{1}{33}=\frac{31}{66}$.
 
In an analogous way, we are able to  derive equivariant and non-equivariant subconvex bounds for $G=\SU(2)$, $K=\SO(2)$, and $\Gamma:=\mklm{\pm 1}$  in the setting of \cite{LPS1,LPS2} by identifying $G$ with the group of units in the quaternion algebra over $\R$, and defining corresponding Hecke operators $T_n$ on $\L^2(\Gamma\bsl G)$. Thus, we obtain again  in Theorem \ref{thm:SO(3)equiv} the equivariant subconvex bound \eqref{eq:11.1.17a} for any simultaneous eigenfunction $\phi_j\in \L^2_{\sigma_l}(\Gamma\bsl G)$ of the Beltrami--Laplace operator  $\Delta$ on $G$  with eigenvalue $\lambda_j$ and the $T_n$, where now $\Gamma\bsl G\cong \SO(3)$. 
This generalizes a result of VanderKam \cite[Theorem 1.1]{VanderKam}, where the case $l=0$ with  $\L^2_{\sigma_0}(\Gamma\bsl G)\cong \L^2(\Gamma \bsl G/K) \cong \L^2(S^2)$ is treated, $S^2$ being the $2$-sphere. 

Our second main result concerns bounds of the form \eqref{eq:M}. As before, let $H$ be a semisimple algebraic group over $\Q$ which is assumed to be connected in the sense of Zariski. Write $\bA_\mathrm{fin}$ for the finite adele ring of $\Q$ and  $\bA:=\R\times\bA_\mathrm{fin}$ for the adele ring.
Choosing an open compact subgroup $K_0$ in $H(\bA_\mathrm{fin})$, we obtain an arithmetic subgroup $\Gamma:=H(\Q)\cap(H(\R) K_0)$ in the semisimple Lie group $G=H(\R)$. 
Assume that $H(\bA)=H(\Q)(H(\R)K_0)$  and that $H(\Q)\bsl H(\bA)$ is compact, so that $\Gamma\bsl G$ is also compact.\footnote{Note that $H(\bA)=H(\Q)(H(\R)K_0)$ is satisfied for any $K_0$ if $H$ has the strong approximation property. In this case, $G=H(\R)$ can be any of the groups $\SL(n;\R), \SL(n,\C), \SL(n,\bH), \SU(m,n,\R)$, $n\geq 2,m\geq 1$, and their products, compare \cite[Section 2.3 and Theorem 7.12]{PR}.}
From the point of view of automorphic representations, one has a suitable family of Hecke operators on $\L^2(\Gamma\bsl G)$, which is given by unramified Hecke algebras over $\Q_p$ for infinitely many primes $p$  \cite{Marshall2017}.
Now, let $K$ be a maximal compact subgroup of $G$ and $\mklm{\phi_j}_{j\geq 0} $  an  orthonormal basis of $\L^2(\Gamma \bsl G)$ consisting of Hecke--Maass forms of rank $1$ with respect to  an elliptic  left-invariant differential operator $P_0$ on $\Gamma \bsl G$ of order $m$ which commutes with the right regular representation of $K$. Assume that $P_0$ is positive and symmetric, and that the cosphere bundle defined by its principal symbol is strictly convex. Then, assuming the condition (WS) made in \cite{Marshall2017},  we show in Theorem \ref{thm:general1} that there exists a constant $\delta>0$ independent of $\sigma$ such that one has the \emph{equivariant subconvex bound} 
\bq
\label{eq:11.1.17b}
\|\phi_j\|_\infty \ll \, \sqrt {d_\sigma \sup_{u \leq \big \lfloor\frac{\dim K}2+1\big\rfloor} \norm{D^u \sigma}_\infty} \, \lambda_j^{\frac{\dim G/K-1}{2m}-\delta}, \qquad \phi_j \in \L^2_\sigma(\Gamma \bsl G),
\eq
where $\lambda_j$ denotes the spectral eigenvalue of $\phi_j$ with respect to $P_0$; if $K=T$ is a torus, one has the stronger estimate
\[
\|\phi_j\|_\infty \ll  \, \lambda_j^{\frac{\dim G/K-1}{2m}-\delta}, \qquad \phi_j \in \L^2 (\Gamma \bsl G).
\]
The bound \eqref{eq:11.1.17b}  bound  sharpens the bound \eqref{eq:R} for a large class of examples. If $\sigma$ is trivial, it is implied by  \eqref{eq:M}.   An example would be given by $H=\SL(1,D)$, where $D$ is any central division algebra of index $n$ over $\Q$, and $G=\SL(n,\R)$.
Furthermore, we show in Theorem \ref{thm:28.5.2017} for some $\delta >0$ the weaker \emph{non-equivariant subconvex bound}
\bq
\label{eq:28.5.2017b}
\|\phi_j\|_\infty \ll \lambda_j^{\frac{\dim G-1}{2m}-\delta}, \qquad \phi_j \in \L^2(\Gamma \bsl G),
\eq
for an orthonormal basis of $\L^2(\Gamma \bsl G)$ consisting of suitable Hecke--Maass forms, sharpening the bound \eqref{eq:ALH}, but without assuming the condition (WS) of \cite{Marshall2017}.  An example is again $H=\SL(1,D)$, where now $D$ is any central division algebra over $\Q$,  except when $G=\SL(1,\bH)$. As before, \eqref{eq:11.1.17b} and \eqref{eq:28.5.2017b} constitute first arithmetic subconvex bounds  on a large class of manifolds which are both positively and negatively curved, and  if $P_0$ is the Beltrami--Laplace operator,  the bounds can be rephrased in terms of the  eigenvalues of the Casimir operator of $G$. Indeed, by Theorem \ref{thm:26.9.2017} we have  for each $\phi_j\in \L^2_\sigma(\Gamma \bsl G)$ with Casimir eigenvalue $\mu_j$ the bound
\[
\|\phi_j\|_\infty \ll \sqrt {d_\sigma \sup_{u \leq \lfloor \frac{\dim K}2+1\rfloor} \norm{D^u \sigma}_\infty} \, \, (-\mu_j+2\mu_\sigma)^{\frac{\dim G/K-1}{4}-\delta}, \qquad \phi_j \in \L^2_\sigma(\Gamma \bsl G),
\]
provided that  $H=\mathrm{Res}_{F/\Q}\underline{G}$ and {\rm (WS)} is fulfilled, while in general
\[
\|\phi_j\|_\infty \ll  (-\mu_j+2\mu_\sigma)^{\frac{\dim G-1}{4}-\delta}, 
\]
  $\mu_\sigma$ being the eigenvalue of the Casimir operator of $K$ on $\sigma$.  If $K=T$ is a torus,
    \[
\|\phi_j\|_\infty \ll  \, (-\mu_j+2\mu_\sigma)^{\frac{\dim G/K-1}{4}-\delta}.
\]

\medskip

Let us briefly say a few words about the methods employed. While in the theory of automorphic forms representation-theoretic tools  prevail, our analysis is mainly based on the spectral theory of elliptic operators, and uses Fourier integral operator methods. Thus, let  $P$ be an elliptic pseudodifferential operator on a closed Riemannian manifold $M$  as above. Our main tool  is the \emph{spectral function} $e(x,y,\mu)$ of the $m$-th root $Q:=\sqrt[m]{P}$ of $P$ given by
\bqn
e(x,y,\mu):=\sum_{\mu_j\leq \mu} \phi_j(x) \overline{\phi_j(y)} \in \Cinft(M \times M), \qquad \mu \in \R, \quad \mu_j:=\sqrt[m]{\lambda_j}.
\eqn
In the spherical situations  \cite{iwaniec-sarnak95,BM1,BM2,BP, Marshall2017} examined before, a crucial role is played by asymptotics for spherical functions, see \cite[Eq. (1.3)]{iwaniec-sarnak95} and \cite[Eq. (8)]{Marshall2017}. Since we cannot rely on them in our setting\footnote{Compare  also \cite[Section 4.4]{MatzProc}.}, we consider instead  the spectral expansion of $e(x,y,\mu)$ itself and the  asymptotic behaviour   of
$$
s_\mu(x,y):=e(x,y,\mu+1) - e(x,y,\mu),
$$
which represents the Schwartz kernel of  the spectral projection $s_\mu $ onto the sum of eigenspaces of $Q$ with eigenvalues in the interval  $(\mu, \mu+1]$. More precisely, if $M$ carries an effective and isometric action of a compact Lie group $K$ and  $\sigma \in \widehat K$, denote by $\Pi_\sigma$ the projector onto the $\sigma$-isotypic component in the Peter-Weyl decomposition of  $\L^2(M)$. In order to show the $\L^\infty$-bounds \eqref{eq:R}, and analogous equivariant convex $\L^p$-bounds, an asymptotic formula for the Schwartz kernel of $s_\mu \circ \Pi_\sigma$, or rather of $\widetilde s_\mu \circ \Pi_\sigma$, where $\widetilde s_\mu$ represents certain smooth approximation to $s_\mu$,  was derived  in \cite[Corollary 2.2. and Theorem 3.3]{ramacher16} in a neighbourhood of the diagonal relying on the theory of Fourier integral operators. Now, let $G$ be a   semisimple Lie group with finite center, $\Gamma $ a discrete cocompact subgroup,  and  $K$ a maximal compact subgroup of $G$. Let $\widehat \Gamma$ denote the set consisting of characters of $\Gamma$ of finite order. For $\chi \in \widehat \Gamma$,  introduce on $\L^2(\Gamma_\chi \bsl G)$  the Hecke operators   $\T^\chi_{\Gamma \beta \Gamma}$ 
\bqn
 (\T^\chi_{\Gamma\beta\Gamma}f)(x):=[\Gamma:\Gamma_\chi]^{-1}\sum_{\alpha\in\Gamma_\chi\bsl\Gamma\beta\Gamma} \overline{\chi(\alpha)}\,  f(\alpha \cdot x),
\eqn
where $\beta$ belongs to a certain set containing the commensurator $C(\Gamma)$ of $\Gamma$. Based on the asymptotics for the kernel of ${\widetilde s_\mu \circ \Pi_\sigma}$ mentioned above,  we deduce in Proposition \ref{thm:15.09.2016} for any  small $\delta >0$ and some constant $C>0$ the equivariant bound
\begin{gather*}
K_{\T^\chi_{\Gamma \beta \Gamma} \circ \tilde s_\mu \circ \Pi_\sigma}(x,x)\ll  \frac{d_\sigma}{[\Gamma:\Gamma_\chi]} \,  \mu^{\dim G/K-1}   \sup_{u \leq \lfloor \frac{\dim K}2+1\rfloor}\norm{D^u \sigma}_\infty  M(x,\beta,\delta)\\  +  \frac{d_\sigma}{[\Gamma:\Gamma_\chi]} \, \mu^{\frac{\dim G/K -1}2}  \sup_{u \leq \big \lfloor\frac{\dim K}2+1\big\rfloor} \norm{D^u \sigma}_\infty \int_\delta^{C} s^{-\frac 12} dM(s)  
\end{gather*}
uniformly in $x \in \Gamma_\chi \bsl G$ for the Schwartz kernel of $\T^\chi_{\Gamma \beta \Gamma} \circ \tilde s_\mu \circ \Pi_\sigma$, where we introduced the lattice point counting function 
\begin{align*}
M(\delta):=M(x,\beta,\delta)&:=\# \mklm{\alpha \in \Gamma_\chi\bsl\Gamma\beta\Gamma\mid  \, \dist(xK,\alpha \cdot xK)^{\dim G/K-1} <\delta }
\end{align*}
given in terms of the distance function on the Riemannian symmetric space $G/K$. In case that $K=T$ is a torus, a corresponding better estimate holds.
From this, we obtain the subconvex bounds \eqref{eq:11.1.17a} and \eqref{eq:11.1.17b} by using known uniform upper bounds \cite{iwaniec-sarnak95,Marshall2017} for $M(x,\beta,\delta)$ combined with arithmetic amplification. 
 The bound \eqref{eq:28.5.2017b} is inferred by analogous methods. In both cases, it is crucial to control the  caustic behaviour of the kernels of ${\tilde s_\mu \circ \Pi_\sigma}$ and ${\tilde s_\mu}$  near the diagonal as $\mu \to +\infty$, respectively.

\medskip

Let us close this introduction with some comments. There  exist several variants of the bounds \eqref{eq:iwaniec-sarnak95}, beginning with  \cite[Appendix]{iwaniec-sarnak95},  where the non compact hyperbolic surface $\SL(2,\Z)\bsl \bH$ is considered. On the other hand, bounds  in  the level aspect  are shown in  \cite{TemplierSelecta} for compact locally symmetric spaces of arithmetic type, while    bounds in the  eigenvalue and level aspect are derived for the modular surfaces $\Gamma_0(N)\bsl \bH$  in \cite{BH,Templier} and other papers. It is likely that our work can be extended to these settings, and we plan to deal with these questions in a future paper.  Also, we intend to  
 widen our results  to Hecke--Maass forms of rank $r$, that is, simultaneous eigenfunctions of the Hecke operators and the full ring of invariant differential operators associated to the center of the universal envelopping algebra of the complexification of the Lie algebra of $G$. For such forms,  the exponent $-1/2m$ in  \eqref{eq:11.1.17b} and \eqref{eq:28.5.2017b} should be improvable by a factor $r$. Finally, we expect the factor $\sqrt {d_\sigma \sup_{u \leq \big \lfloor\frac{\dim K}2+1\big\rfloor} \norm{D^u \sigma}_\infty}$  in  \eqref{eq:11.1.17b} to be improvable to   $d_\sigma$.

\medskip

This paper is structured as follows. In Section \ref{sec:hecke}, we introduce Hecke operators with character on semisimple Lie groups with finite center, in Section \ref{sec:FIO} we give a  description of the asymptotic behaviour of  spectral function of an elliptic operator by means of  Fourier integral operators, and explain how  convex bounds  can be deduced from this  in equivariant and non-equivariant situations.  Based on these results, we derive  spectral asymptotics for kernels of Hecke operators  in Section \ref{sec:4}. Relying on the latter, we finally prove    subconvex bounds for arithmetic congruence lattices in $\SL(2,\R)$, $\SO(3)$,  and a large class of  semisimple algebraic groups in Sections \ref{sec:IS}, \ref{sec:SO(3)},  and \ref{sec:6}, respectively. Throughout the paper, $\N:=\mklm{0,1,2,3,\dots}$ will denote the set of natural numbers, while $\N_\ast:=\mklm{1,2,3,\dots}$. 

\medskip

{\bf Acknowledgements}. We would like to thank Valentin Blomer and Simon Marshall for their comments on an earlier draft of this paper.  Besides, the second author would like to thank Jasmin Matz for her advice on spherical functions. He is partially supported by the JSPS Grant-in-Aid for Scientific Research (No. 15K04795, 18K03235).

\section{Hecke operators with character on semisimple Lie groups}
\label{sec:hecke}

To introduce our setting, let $G$ be a real semisimple Lie group  with finite center  and Lie algebra $\g$.  Denote by  $\langle X,Y\rangle := \tr \, (\ad X\circ \ad Y)$ the Cartan-Killing form   on $\g$ and by $\theta$   a Cartan involution  of $\g$. Let 
\bqn
\label{eq:cartan}
\g = \k\oplus\p
\eqn
be  the Cartan decomposition of $\g$ into the eigenspaces of  $\theta$, corresponding to the eigenvalues  $+1$ and $-1$ , respectively, and denote the  maximal compact subgroup of $G$ with Lie algebra $\k$ by $K$. Put $\langle X,Y\rangle _\theta:=-\langle X,\theta Y\rangle $. Then $\langle \cdot,\cdot \rangle _\theta$ defines a left-invariant Riemannian metric on $G$,  which in general will possess some strictly positive sectional curvature,  compare Milnor \cite[p.\ 298 and p.\ 317]{milnor76}. Dividing by the $K$-action,  the quotient $G/K$ becomes a Riemannian symmetric space of non-positive sectional curvature.   With respect to the left-invariant metric on $G$, a distance function   $\dist(g,h)$ is defined  on each connected component of $G$  as  the geodesic distance  between two points $g,h$ in that component. Note that $\dist(g_1g,g_1h) = \dist(g,h)$ for all $g_1\in G$.  In contrast to the Killing form,  $\langle \cdot,\cdot \rangle _\theta$ is no longer $\Ad(G)$-invariant, but  still $\Ad(K)$-invariant, so that $\dist (gk,hk)=\dist (g,h)$ for all $k \in K$.

Next,  let $X_1, \dots X_{\dim \p}$ be an orthonormal basis of $\p$ and $Y_1, \dots, Y_{\dim \k}$ an orthonormal basis for $\k$ with respect to $\langle \cdot ,\cdot \rangle _\theta$. If $\Omega$ and $\Omega_K$ denote the Casimir elements of $G$ and $K$, one has
 \bq
 \label{eq:casimir} 
 \Omega= \sum _{i=1}^{\dim \p} X_i^2- \sum _{i=1}^{\dim \k} Y_i^2, \qquad \Omega_K=-\sum _{i=1}^{\dim \k} Y_i^2,
 \eq
and we put $\Theta:=-\Omega+2\Omega_K$. 
 Then  $dR(\Theta)$ is the Beltrami--Laplace operator $\Delta$ on $G$ with respect to the left invariant metric defined by $\langle \cdot, \cdot \rangle_\theta$,  while $\Ccal:=dR(\Omega)$ represents the Casimir operator,  $R$ being the right regular representation of $G$ on $\Cinft(G)$, 
see \cite[Section 3]{mueller98} and \cite[Section 2.10]{borel97}.  Thus, 
\bq
\label{eq:casimir-laplace}
\Delta=-\Ccal+2dR(\Omega_K),
\eq
all three operators commuting with each other.

We consider now  a discrete  cocompact subgroup $\Gamma$ of $G$, together with the set $\widehat \Gamma$ of its characters of finite order, and let $\chi\in \widehat \Gamma$. 
Then $\Gamma_\chi:= \ker \chi$ is a subgroup of finite index in $\Gamma$ and the quotient  $\Gamma_\chi \backslash G$ a compact manifold without boundary. By requiring that the projection $G \rightarrow \Gamma_\chi \bsl G$ is a Riemannian submersion, we obtain a Riemannian structure on $\Gamma_\chi \backslash G$ which locally has the same curvature than $G$. Furthermore, the Riemannian structure on $G/K$ induces a Riemannian metric on $\Gamma_\chi \bsl G/K$, becoming a   locally symmetric space of negative curvature. In both cases, $\dist$ induces corresponding distances on $\Gamma_\chi \bsl G$ and $\Gamma_\chi \bsl G/K$. 

Before we proceed, note that due to the compactness of  $\Gamma_\chi\bsl G$,  the right regular representation of $G$ on $\L^2(\Gamma_\chi \bsl G)$ decomposes into an orthogonal direct sum of countably many irreducible unitary representations  with finite multiplicities, that is,
\bq
\label{eq:Gdecomp}
\L^2(\Gamma_\chi \bsl G)\cong \bigoplus_{\pi\in\widehat{G}} m(\pi,\Gamma_\chi)\cdot \pi,
\eq
where $\widehat{G}$ denotes the unitary dual of $G$ and $m(\pi,\Gamma_\chi)$ is a non-negative integer, see \cite{GGP}.
Furthermore, both the spectra of $\Delta$ and $\mathcal{C}$ in $\L^2(\Gamma_\chi \bsl G)$ are discrete. Note that the eigenvalues of $\Delta$ are positive, while the ones of $\Ccal$ can be both negative and positive.

 In what follows, we introduce Hecke operators on $\Gamma \bsl X$, where $X:=G$ or $G/K$, following \cite[Section 2]{Hoffmann}, and consider the commensurator 
\[
C(\Gamma):=\{ g\in G \mid \text{$\Gamma$ is commensurable with $g^{-1}\Gamma g$} \}
\]
of $\Gamma$,  where we say that two subgroups $\Gamma_1$ and $\Gamma_2$ are \emph{commensurable} iff the indices $[\Gamma_1:\Gamma_1\cap\Gamma_2]$ and $[\Gamma_2:\Gamma_1\cap\Gamma_2]$ are finite. Let $\beta\in C(\Gamma)$. Since the mapping
\[
(\Gamma\cap \beta^{-1}\Gamma\beta)\bsl\Gamma \ni (\Gamma\cap \beta^{-1}\Gamma\beta)\gamma \mapsto \Gamma\beta\gamma \in \Gamma\bsl \Gamma \beta \Gamma
\]
is bijective, the double coset $\Gamma \beta \Gamma$ is a finite union of right cosets of $\Gamma$, that is, there exist representative elements $\beta_1$, $\beta_2,\dots,\beta_t$ in $\Gamma\beta\Gamma$ such that
\[
\Gamma\beta\Gamma=\bigsqcup_{j=1}^t\Gamma \beta_j.
\]
One can then associate to each double coset a linear operator $T_{\Gamma\beta\Gamma}$ on $\L^2(\Gamma \bsl X)$ by setting 
\[
T_{\Gamma\beta\Gamma}: \L^2(\Gamma \bsl X) \longrightarrow \L^2(\Gamma \bsl X), \quad   (T_{\Gamma\beta\Gamma}f)(x):=\sum_{j=1}^t f(\beta_j \cdot x),
\]
 where $\beta_j \cdot  x \equiv \Gamma \beta_j \cdot  \Gamma x:= \Gamma \beta_j x$ depends on the choice of the representative $x$, but the  sum does not depend on the choice of the representatives $x$ and   $\beta_j$. Summing up, one  writes\footnote{Here and above $\alpha \equiv \Gamma \alpha$ (resp. $\beta_j\equiv \Gamma\beta_j $) and $x \equiv \Gamma x$ are considered both as right cosets of $\Gamma$ and representatives in $G$, and the products  $\alpha x$ (resp. $\beta_j x$) are taken in $G$.}
\[
(T_{\Gamma\beta\Gamma}f)(x)=\sum_{\alpha\in \Gamma\bsl \Gamma\beta\Gamma} f(\alpha \cdot  x),
\]
and calls   $T_{\Gamma\beta\Gamma}$  a \emph{Hecke operator}.

 If a subset $U$ of $C(\Gamma)$ is decomposed into a finite disjoint union of double cosets of $\Gamma$,  a linear operator $T_U$ can be  defined in the same manner according to 
\bq
\label{eq:heckeop}
T_U:=\sum_{k=1}^uT_{\Gamma\beta_k\Gamma} \qquad \quad U=\bigsqcup_{k=1}^u \Gamma\beta_k\Gamma,  \quad \beta_k\in C(\Gamma).
\eq
More generally, one can introduce Hecke operators as follows. 
Write $H(\Gamma,C(\Gamma))$ for the space of left and right $\Gamma$-invariant $\C$-valued functions $h$ on $C(\Gamma)$ such that the support of $h$ is included in a finite union of double $\Gamma$-cosets. Endowed with the convolution product
\[
h_1*h_2(x):=\sum_{y\in \Gamma\bsl C(\Gamma)}h_1(y)\, h_2(xy^{-1}), \qquad h_1,h_2\in H(\Gamma,C(\Gamma)),
\]
 $H(\Gamma,C(\Gamma))$ becomes an associative algebra over $\C$ with the characteristic function $1_\Gamma$ of $\Gamma$ as  unit element.
For each $h \in H(\Gamma,C(\Gamma))$, a linear operator $T_h$ on $\L^2(\Gamma \bsl X)$ can then be  defined by
\[
(T_h f)(x):=\sum_{\alpha\in\Gamma\bsl C(\Gamma)}h(\alpha) \, f(\alpha\cdot x),
\]
and one has $T_{h_1*h_2}=T_{h_1}\circ T_{h_2}$.
If $U$ is as above and $h$ is the characteristic function of $U$, then it is obvious that $T_h$ equals $T_U$.
We call $H(\Gamma,C(\Gamma))$ the \emph{Hecke algebra} and refer the reader to \cite[Section 2]{Hoffmann} for details. Next, let us introduce Hecke operators with character.
Assume that $G$ is a subgroup of another group $G'$.
Let $\Xi$ be a sub-semigroup of $G'$ containing $\Gamma$.
We suppose that $\chi$ can be extended to $\Xi$, and that there exists a homomorphism $\psi:\Xi\to C(\Gamma)$ such that $\psi|_\Gamma$ is the identity map, and $\alpha x\alpha^{-1}=\psi(\alpha)x\psi(\alpha)^{-1}$ holds for any $\alpha\in\Xi$, $x\in G$.
In particular, for $\alpha$, $\beta\in\Xi$ we have $\chi(\alpha\beta)=\chi(\alpha)\, \chi(\beta)$, while the inverse element $\alpha^{-1}$ does not always belong to $\Xi$.

\begin{ex}
One of the main examples we are having in mind is $G':=\GL(n,\R)$ with 
\[
 \Xi:=\{ \alpha=(\alpha_{ij})\in M(n,\Z)\mid  \det(\alpha)>0, \;\;   (\alpha_{11},N)=1, \; \; \alpha_{j1}\equiv 0 \mod N \;\; (2\leq j\leq n)  \}, 
\]
\[
\chi((\alpha_{ij})):=\omega(\alpha_{11}),\quad G:=\SL(n,\R), \quad \Gamma:=G\cap \Xi,\quad \psi(\alpha):=\det(\alpha)^{-1/n}\alpha,
\]
where $\omega$ is a Dirichlet character on $(\Z/N\Z)^\times$.
\end{ex}

Let us now define a  left action of $\Xi$ on $G$ by setting $\alpha\cdot x:=\psi(\alpha)x$. For a fixed $\beta \in \Xi$ we can then define the \emph{Hecke operator with character}
\bq
\label{eq:heckechar0}
\T^\chi_{\Gamma\beta\Gamma}: \L^2(\Gamma_\chi \bsl X) \longrightarrow \L^2(\Gamma_\chi  \bsl X), \quad   (\T^\chi_{\Gamma\beta\Gamma}f)(x):=[\Gamma:\Gamma_\chi]^{-1}\sum_{\alpha\in\Gamma_\chi\bsl\Gamma\beta\Gamma} \overline{\chi(\alpha)}\,  f(\alpha \cdot x),
\eq
where we took into account that $\Gamma_\chi  \bsl \Gamma \beta \Gamma\subset \Xi$. By  definition we have $\T^\chi_{\Gamma\beta\Gamma}=T_h$ for some $h\in H(\Gamma_\chi,C(\Gamma))$ satisfying $h(\gamma_1 x \gamma_2)=\overline{\chi(\gamma_1\gamma_2)}h(x)$ for any  $\gamma_1,\gamma_2\in\Gamma$ and  $x\in C(\Gamma)$. Furthermore, for given  $\beta_j\in\Xi$  and $h_j\in H(\Gamma_\chi,C(\Gamma))$ with  $T_{h_j}=\T^\chi_{\Gamma\beta_j\Gamma}$,  the convolution $h_1*h_2$ also satisfies the latter condition.
Thus, there exist elements $l\in\N$, $a_u\in\C$, and $\alpha_u\in\Xi$ such that
\bq
\label{eq:16.2.2017}
\T^\chi_{\Gamma\beta_1\Gamma}\circ \T^\chi_{\Gamma\beta_2\Gamma}=\sum_{u=1}^l a_u \T^\chi_{\Gamma\alpha_u\Gamma}.
\eq
Next,  denote by  $\L^2_\chi(\Gamma \bsl X)$ the Hilbert space of measurable functions on $X$ such that 
\bq
\label{eq:20.11.2016}
f(\gamma x)=\chi(\gamma)\, f(x),  \qquad \gamma\in\Gamma,\quad x\in X,
\eq
and 
\bq 
\label{eq:21.11.2016}
\norm{f}:=\bigg (\int_{\Gamma \bsl X} |f(x)|^2 dx\bigg ) ^{1/2} < \infty,
\eq
which is well-defined since $|\chi(\gamma)|=1$ for $\gamma \in \Gamma$, compare \cite[p. 228]{Miyake}.\footnote{Note that for $\gamma \in \Gamma_\chi$ condition  \eqref{eq:20.11.2016} reads $f(\gamma x) =f(x)$. Therefore, instead of $\L^2_\chi(\Gamma \bsl X)$ one could also consider the closed subspace  of $\L^2(\Gamma_\chi \bsl X)$ that consists of functions satisfying  \eqref{eq:20.11.2016}.} Notice that $f \in \L^2_\chi(\Gamma \bsl X)$ implies $|f| \in \L^2(\Gamma \bsl X)$. 
If $\chi$ is trivial, then  $\L^2_\chi(\Gamma \bsl X)=\L^2(\Gamma \bsl X)$. Since $\Gamma_\chi$ is a normal subset of $\Gamma$ we have
\bq
\label{eq:directsum}
\L^2(\Gamma_\chi \bsl X)\cong \bigoplus_{\chi'  \in \, \widehat{\Gamma/\Gamma_\chi}}\L^2_{\chi'}(\Gamma \bsl X),
\eq
where we regard $\widehat{\Gamma/\Gamma_\chi}$ as a subset of $\widehat{\Gamma}$, compare \cite[Lemma 4.3.1]{Miyake}. In particular, because $\chi \in \widehat{\Gamma/\Gamma_\chi}$,   $\L^2_\chi(\Gamma \bsl X)$ is  a closed subspace in $\L^2(\Gamma_\chi \bsl X)$, and 
for a fixed $\beta \in \Xi$, the operator $\T^\chi_{\Gamma\beta\Gamma}$ restricts to the linear operator
\bq
\label{eq:heckechar}
T^\chi_{\Gamma\beta\Gamma}: \L^2_\chi(\Gamma \bsl X) \longrightarrow \L^2_\chi(\Gamma \bsl X), \quad   (T^\chi_{\Gamma\beta\Gamma}f)(x):=(\T^\chi_{\Gamma\beta\Gamma}f)(x)=\sum_{\alpha\in\Gamma\bsl\Gamma\beta\Gamma} \overline{\chi(\alpha)}\,  f(\alpha \cdot x).
\eq
Notice that for each $\chi'$ in $\widehat{\Gamma/\Gamma_\chi}$ with $\chi'\neq \chi$ and each function $f \in \L^2_{\chi'}(\Gamma\bsl X)$ we have
\bq
\label{eq:vanish}
\T_{\Gamma\beta\Gamma}^\chi f(x)=[\Gamma:\Gamma_\chi]^{-1}\sum_{\alpha_1\in\Gamma_\chi\bsl\Gamma}
\sum_{\alpha_2\in\Gamma\bsl\Gamma\beta\Gamma}
\overline{\chi(\alpha_1\alpha_2)}\chi'(\alpha_1) f(\alpha_2\cdot x)=0
\eq
by the orthogonality relations for characters. 
Further,  the projection of $\L^2(\Gamma_\chi \bsl X)$ onto $\L^2_{\chi}(\Gamma \bsl X)$ is given by the Hecke operator $T_{h(\chi)}$,  where  $h(\chi)\in H(\Gamma_\chi, \Xi)$ is the function  
\bq
\label{eq:heckeproj}
h(\chi): \,  x \, \longmapsto \, [\Gamma:\Gamma_\chi]^{-1} \sum_{\alpha\in \Gamma_\chi\bsl \Gamma} \overline{\chi(\alpha)} \, \1_{\Gamma_\chi\alpha}(x),
\eq
$\1_{\Gamma_\chi\alpha}$ being the characteristic function of the coset ${\Gamma_\chi\alpha}$.
Thus, one obtains the commutative diagram

\medskip

\bq
\label{diagram}
\begin{CD}
\L^2(\Gamma_\chi\bsl X) @>\T^\chi_{\Gamma \beta \Gamma}>> \L^2(\Gamma_\chi\bsl X) \\
@VVT_{h(\chi)} V @VVT_{h(\chi)}V \\
\L^2_\chi(\Gamma\bsl X) @>T^\chi_{\Gamma \beta \Gamma}>> \L^2_\chi(\Gamma\bsl X)
\end{CD}
\eq

\bigskip

\noindent
and in view of \eqref{eq:vanish} we have  
$
T_{h(\chi)} \circ  T^\chi_{\Gamma \beta \Gamma} \circ T_{h(\chi)}= \T^\chi_{\Gamma \beta \Gamma}. 
$

\section{The  spectral function of an elliptic operator and convex bounds for eigenfunctions}
\label{sec:FIO}

The main tool underlying our analysis is the spectral function of an elliptic operator on a smooth manifold, which contains essential  information on the spectrum. For large spectral parameters, an asymptotic description of it can be derived within the theory of Fourier integral operators, yielding in particular convex bounds for eigenfunctions. In what follows, we shall briefly recall the main arguments in non-equivariant and equivariant situations, and provide the results that will be needed later.

\subsection{The spectral function and convex bounds for eigenfunctions}\label{sec:spec.funct}

Let $M$ be  a closed  Riemannian manifold of dimension $d$  and $P_0$ an elliptic classical pseudodifferential operator on $M$ of degree $m$, which is assumed to be  positive and symmetric. Denote its  unique self-adjoint extension by  $P$, and let $\mklm{\phi_j}_{j\geq 0}$ be an orthonormal basis of $\L^2(M)$ consisting of eigenfunctions of $P$ with eigenvalues $\mklm{\lambda_j}_{j \geq 0}$ repeated according to their multiplicity. 
Let $p(x,\xi)$ be the principal symbol of $P_0$, which is strictly positive and  homogeneous in $\xi$ of degree $m$ as a function  on $T^\ast M\setminus\mklm{0}$, that is,  the cotangent bundle of $M$ without the zero section.  Here and in what  follows  $(x,\xi)$ denotes an element in $T^*Y \simeq Y \times \R^d$ with respect to the canonical trivialization of the cotangent  bundle over a chart domain $Y\subset M$. Consider further the $m$-th root $Q:=\sqrt[m]{P}$ of $P$ given by the spectral theorem. It is well  known that $Q$ is a classical pseudodifferential operator of order $1$ with principal symbol $q(x,\xi):=\sqrt[m]{p(x,\xi)}$ and the first Sobolev space as domain. Again, $Q$ has discrete spectrum, and its eigenvalues  are given by $\mu_j:=\sqrt[m]{\lambda_j}$.  The spectral function $e(x,y,\lambda)$ of $P$ can then be described by  studying the spectral function of $Q$, which in terms of the basis $\mklm{\phi_j}$ is given by
\bq
\label{eq:specfunct} 
e(x,y,\mu):=\sum_{\mu_j\leq \mu} \phi_j(x) \overline{\phi_j(y)},
\eq
and belongs to $\Cinft(M \times M)$ as a function of $x$ and $y$ for any $\mu \in \R$. Let $s_\mu$ be the spectral projection onto the sum of eigenspaces of $Q$ with eigenvalues in the interval  $(\mu, \mu+1]$, and denote its Schwartz kernel by 
$$
s_\mu(x,y):=e(x,y,\mu+1) - e(x,y,\mu).
$$
 To obtain an asymptotic description of the spectral function of $Q$, one first derives a description of $s_\mu(x,y)$ by approximating $s_\mu$ by Fourier integral operators. To do so, let $\rho \in \S(\R,\R_+)$ be such that $\hat \rho(0)=1$ and $\supp \hat \rho\in (-\delta/2,\delta/2)$ for an arbitrarily small  $\delta>0$, and define the {approximate spectral projection operator} 
\bqn 
\widetilde s_\mu u := \sum_{j=0}^\infty \rho(\mu-\mu_j) E_{j} u, \qquad u \in \L^2(M),
\eqn
where $E_j$ denotes the orthogonal projection onto the subspace spanned by $\phi_j$. Clearly, 
\bq
\label{eq:29.5.2017}
K_{\widetilde s_\mu}(x,y):=\sum_{j=0}^\infty \rho(\mu-\mu_j) \phi_j(x) \overline{\phi_j(y)}\in \Cinft(M\times M)
\eq
 constitutes the kernel of $\widetilde s_\mu$. Now, notice that for $\mu,\tau\in \R$ one has
\bqn 
\rho(\mu -\tau) = \frac 1 {2\pi} \intop_\R \hat \rho(t) e^{-it\tau} e^{it\mu} \d t, 
\eqn
where $\hat \rho(t)$ denotes the Fourier transform of $\rho$, 
so that for $u \in \L^2(M)$ we obtain
\begin{align}
\label{eq:19.1.2017}
\widetilde s_\mu u=\frac 1 {2\pi} \sum_{j=0}^\infty \intop_\R \hat \rho(t)  e^{it\mu}  e^{-it\mu_j}\d t \, E_j u= \frac 1 {2\pi}  \intop_\R \hat \rho(t)  e^{it\mu}  U(t) u \d t, 
\end{align}
where $U(t)$ stands for  the one-parameter group 
  \bqn 
 U(t):=e^{-itQ}=\int e^{-it\mu} dE_\mu^Q, \qquad t \in \R,
 \eqn
of unitary operators in $\L^2(M)$  given by the Fourier transform of the spectral measure, 
 $\{E_\mu^Q\}$ being a spectral resolution of $Q$. The central result of H\"ormander \cite{hoermander68}  then says that  $U(t):\L^2(M)\rightarrow \L^2(M)$ can be approximated by Fourier integral operators.
 
 More precisely, let $\mklm{(\kappa_\iota, Y_\iota)}_{\iota \in I}$, $\kappa_\iota:Y_\iota \stackrel{\simeq}\to \widetilde Y_\iota \subset \R^d$, be an atlas for $M$,  $\mklm{f_\iota}$ a corresponding partition of unity and  
$
\hat v (\eta) :=\F(v)(\eta):= \int_{\R^d} e^{-i\langle \tilde y, \eta \rangle} v(\tilde y) \, d\tilde y$ the Fourier transform of $v \in \CT(\widetilde Y_\iota)$. Write $\dbar \eta:= \d\eta/(2\pi)^d$, and introduce  on $\widetilde Y_\iota$  the operator 
\bqn 
[\widetilde U_\iota(t)v] (\tilde x):= \int _{\R^d} e ^{i\psi_\iota(t,\tilde  x, \eta)} a_\iota (t, \tilde x, \eta) \hat{v}(\eta) \dbar \eta,
\eqn 
 where $a_\iota \in S^0_{\mathrm{phg}}$ is a classical polyhomogeneous symbol satisfying $a_\iota(0,\tilde x, \eta)=1$ and $\psi_\iota$  the defining phase function given as the solution of the Hamilton-Jacobi equation
\bqn 
\frac {\gd \psi_\iota} { \gd t } + q \Big (x, \frac{\gd \psi_\iota}{\gd \tilde x}\Big )=0 , \qquad  \psi_\iota( 0, \tilde x, \eta) = \eklm{\tilde x, \eta},
\eqn  
see \cite[ Page 254]{hoermanderIV}.   Let us remark that  $\psi_\iota$ is homogeneous in $\eta$ of degree $1$, so that Taylor expansion for small $t$ gives
\bqn
\psi_\iota(t,\tilde x, \eta) =\psi_\iota(0,\tilde x, \eta) +t \frac{\gd \psi_\iota}{\gd t} (0, \tilde x , \eta) + O(t^2|\eta|)= \eklm{\tilde x, \eta} -t q_\iota(\tilde x, \eta) +O(t^2|\eta|), 
\eqn 
where we wrote $q_\iota(\tilde x, \eta):=q(\kappa_\iota^{-1}(\tilde x),\eta)$. In other words, there exists a smooth function $\zeta_\iota$ which is homogeneous in $\eta$ of degree $1$ and satisfies
\begin{align*}
\begin{split}
\psi_\iota(t, \tilde x, \eta) &= \eklm {\tilde x, \eta} -t \zeta_\iota(t, \tilde x, \eta), \qquad \qquad \zeta_\iota(0, \tilde x, \eta) = q_\iota(\tilde x, \eta).  
\end{split}
\end{align*}
Put now $\bar U_\iota(t) u := [ \widetilde U_\iota(t) (u \circ \kappa_\iota^{-1})] \circ \kappa_\iota$ for any  $ u \in \CT( Y_\iota)$. Consider further  test functions $\bar f_\iota \in \CT( Y_\iota)$ satisfying $\bar f_\iota \equiv 1$ on $\supp f_\iota$, and define
\bqn 
\bar U(t) := \sum _\iota F_\iota \, \bar U_\iota(t) \, \bar F_\iota, 
\eqn
where  $F_\iota$, $\bar F_\iota$ denote the multiplication operators with  $f_\iota$ and $\bar f_\iota$, respectively. Then  H\"ormander showed that for small $|t|$ 
\bq
\label{eq:R(t)}
\mathcal{R}(t) := U(t) -\bar U(t) \, \text{is an operator with smooth kernel,}
\eq 
compare \cite[ Page 134]{grigis-sjoestrand} and \cite[Theorem 20.1]{shubin}; in particular, the kernel $\mathcal{R}_t(x,y)$ of $\mathcal{R}(t)$ is smooth as a function of $t$. 

Approximating in \eqref{eq:19.1.2017} the operator $U(t)$  by $\bar U(t)$, one obtains a  description for the kernel of $\widetilde s_\mu$ as the double oscillatory integral
\begin{align}
\label{eq:20.01.2017bis}
\begin{split}
K_{\widetilde s_\mu}(x,y)
=&  \frac {\mu^d  }{(2\pi)^{d+1}}  \sum_\iota  \int_{\R}\int _{\R} e^{i\mu[t-Rt]}   I_{\iota}(\mu, R, t, x,y)  \d R\d t 
\end{split}
\end{align}
up to terms of order  $O(\mu^{-\infty})$ which are uniform in $x$ and $y$, where   
\begin{align*}
\begin{split}
 I_{\iota}(\mu, R, t, x,y):= &     \int_{\Sigma^{R,t}_{\iota,x}} e^{i{ \mu}  \Phi_{\iota,x,y}(\omega)} \hat \rho(t)   f_\iota( x) \,     a_\iota(t, \kappa_\iota(  x) , \mu \omega)  \bar f _\iota (y)   b(q( x, \omega)) {\d\Sigma^{R,t}_{\iota,x}(\omega)  }, \\
 \Phi_{\iota,x,y}(\omega):=&\eklm{\kappa_\iota(  x) - \kappa_\iota(y),\omega},
 \end{split}
\end{align*}
and 
\bq
\label{eq:20.04.2015}
\Sigma^{R,t}_{\iota,x}:=\mklm{\omega \in \R^d  \mid \zeta_\iota (t, \kappa_\iota (x),\omega) = R},
\eq
while $0 \leq b  \in \CT(1/2,3/2)$ is a test function such that $b\equiv 1$ in  a neighborhood of $1$, compare \cite[Eq. (2.8)]{ramacher16}. Here ${d\Sigma^{R,t}_{\iota,x}(\omega)}$ denotes the quotient of Lebesgue measure  in $\R^d$ by Lebesgue measure in $\R$ with respect to $\zeta_\iota(t, \tilde x, \omega)$.  Furthermore, for sufficiently small $\delta>0$ one can assume that the $R$-integration is over a compact set, and  $R$ and $t$ are close to $1$ and $0$, respectively.  From \eqref{eq:20.01.2017bis}, an asymptotic description can be inferred as $\mu \to +\infty$ by means of the stationary phase principle. In fact, one has the following

\begin{proposition}
\label{prop:31.05.2017}
Suppose  that the cospheres $S^\ast_xM:=\mklm{(x,\xi) \in T^\ast M\mid  p(x,\xi)=1}$ are strictly convex.\footnote{This condition is required only in the case $x\not=y$. For example, it holds if $P_0=\Delta$ equals the Beltrami--Laplace operator, since then $p(x,\xi)=\|\xi\|^2_x$.}
Then, for any  fixed $x, y \in M$,  and $\tilde N=0,1,2,3,\dots$ one has the expansion
\begin{align*}
\begin{split}
K_{\widetilde s_\mu}(x,y)&=   \mu^{d-1-{\frac{\delta_{x,y}}{2}}} \left [ \sum_{r=0}^{\tilde N-1} \Lcal_r(x,y,\mu) + \cR_{\tilde N}(x,y,\mu) \right ]
\end{split}
\end{align*}
up to terms of order $O(\mu^{-\infty})$ as $\mu \to +\infty$, where  
\bqn 
\delta_{ x,y}:=\begin{cases} 0, & y=x, \\ d-1, & y \not=x.  \end{cases}
\eqn
The coefficients in the expansion and  the remainder $\cR_{\tilde N}(x,y,\mu)=O_{x,y}(\mu^{-\tilde N})$ term can be computed explicitly;  if $y=x$, they  are uniformly bounded in $x$ and $y$,  while  if $y \not=x$, they satisfy the bounds
 \begin{align}
 \label{eq:31.5.2017}
\begin{split}
\Lcal_r(x,y,\mu) &\ll \, {\dist (x,y)}^{-(d-1)/2-r} \, \mu^{-r},  \qquad \cR_{\tilde N}(x,y,\mu) \ll \, \dist (x, y)^{-(d-1)/2-\tilde N} \, \mu^{-\tilde N},
\end{split}
\end{align}
where $\dist(x,y)$ denotes the geodesic distance between two points  belonging to the same connected component. Otherwise,  $\dist(x,y):=\infty$. 
\end{proposition}  

\begin{proof}
If $x=y$, one has $I_{\iota}(\mu, R, t, x,y)=O(1)$ uniformly in all parameters since  $a_\iota \in \mathrm{S}^0_{phg}$ is a classical symbol of order $0$, so that 
\bqn 
\big | \gd ^\alpha_\omega a_\iota ( t, \kappa _\iota (x), \mu \omega) \big | =|\mu|^{|\alpha|} \big | (\gd^\alpha_\omega a _\iota)( t, \kappa_\iota(x), \mu \omega) \big | \leq C |\omega|^{-|\alpha|}.
\eqn
Consequently, the  dependence of the amplitude on $\mu$ does not interfer with the asymptotics, compare \cite[Proposition 1.2.4]{duistermaatFIO}. Applying the stationary phase principle \cite[Theorem 7.7.5]{hoermanderI} to the $(R,t)$-integral in \eqref{eq:20.01.2017bis} with $t(1-R)$ as phase function then yields the assertion for $x=y$, the unique critical point being $(R_0,t_0)=(1,0)$ in this case. Let us now assume that $x\not=y$.  By assumption, the cospheres $S_x^\ast M$ are strictly convex, so that for small $|t|$, the hypersurfaces $\Sigma^{R,t}_{\iota,x}$ will  be strictly convex, too.  Applying \cite[Lemma 3.5]{ramacher16}  to the integrals $I_{\iota}(\mu, R, t, x,y)$ with  $\eklm{\frac{\kappa_\iota(  x) - \kappa_\iota(y)}{\norm{\kappa_\iota(  x) - \kappa_\iota(y)}},\omega}$ as phase function and $\nu:=\mu \norm{\kappa_\iota(  x) - \kappa_\iota( y)} $ as asymptotic parameter yields for any $\tilde N\in \N$ the expansion  
\begin{align*}
\begin{split}
 I_{\iota}(\mu, R, t, x,y)&= \sum_{{\omega_0 \in \Crit \Phi_{\iota, x,y}}} {e^{i\mu\Phi_{\iota, x,y}(\omega_0)}  }\left [ \sum_{r=0}^{\tilde N-1} \mathcal{Q}_{\iota, r}(R,t,x,y,\omega_0) \mu  ^{-(d-1)/2-r}  + \mathcal{R}_{\iota, \tilde N}(R,t,x,y,\omega_0,\mu)\right ],
 \end{split}
\end{align*}
where the coefficients and the remainder are smooth in $R$ and $t$, and satisfy the bounds
\begin{align*}
\begin{split}
\gd^\alpha_{R,t}\mathcal{Q}_{\iota,r}(R,t,x,y,\omega_0) &\ll  \norm{\kappa_\iota(x)-\kappa_\iota(y)}^{ -(d-1)/2-r}, \\
\gd^\alpha_{R,t} \mathcal{R}_{\iota, \tilde N}(R,t,x,y, \omega_0,\mu) & \ll (\norm{\kappa_\iota(x)-\kappa_\iota( y)}\mu)^{-(d-1)/2- \tilde N},
\end{split}
\end{align*}
uniformly in $R$ and $t$. Regarding the value of $\Phi_{\iota, x,y}$ on its critical set, one  computes for $\omega_0 \in  \Crit \, \Phi_{\iota, x,y}$
\bqn
\Phi_{\iota, x,y}(\omega_0)=\pm   {\norm {\kappa_\iota(x) - \kappa_\iota( y )}} R /{\norm {\grad_\eta \zeta_\iota(t,\kappa_\iota(x),\omega_0) }},
\eqn
since $\kappa_\iota(x) - \kappa_\iota( y )$ must be colinear to $\grad_\eta \zeta(t,\kappa_\iota(x),\omega_0)$. Notice that due to the fact that $\zeta(t,\kappa_\iota(x),\eta)$ is homogeneous of degree $1$ in $\eta$, the gradient $\grad_\eta \zeta(t,\kappa_\iota(x),\omega_0)$ only depends on the direction of $\omega_0$, and is therefore independent of $R$. From this and  \eqref{eq:20.01.2017bis} we deduce for $K_{\widetilde s_\mu }(x,y)$ as $\mu \to +\infty$ the expansion
\begin{align*}
\begin{split}
  \frac {\mu^d   }{(2\pi)^{d+1}}  \sum _{\iota, \omega_0}  \int_{\R}\int _{\R} e^{i\mu[t-Rt+\Phi_{\iota, x,y}(\omega_0)]}  \left [ \sum_{r=0}^{\tilde N}  \mathcal{Q}_{\iota,r}(R,t,x,y,\omega_0) \mu^{-(d-1)/2-r}  + \mathcal{R}_{\iota,\tilde N}(R,t,x,y,\omega_0,\mu)\right ]  d R\d t
\end{split}
\end{align*}
up to terms of order $O(\mu^{-\infty})$. Again, we apply the stationary phase principle to the $(R,t)$-integrals, where now the phase function reads $t(1-R) +\Phi_{\iota,x,y}(\omega_0)$. The determinant of the matrix of its second derivatives is given by
\bqn
- \left (1  \pm O({\norm {\kappa_\iota(x) - \kappa_\iota(y)}})    \right )^2.
\eqn
  By choosing the charts $Y_\iota$ sufficiently small so that $\|\kappa_\iota(x) - \kappa_\iota(y)\| \ll 1$, we can therefore achieve that  in a  sufficiently small neighborhood of $(R,t)=(1,0)$, which is where the amplitude of the $(R,t)$-integral is supported, the phase function $t(1-R) +\Phi_{\iota,x,y}(\omega_0)$ has, if at all, only non-degenerate, hence isolated, critical points.  If we now apply the stationary phase theorem, the assertion follows in the case $x\not=y$ as well. 
\end{proof}

\begin{rem}
\label{rem:11.06.2017}
By Cauchy-Schwarz and the positivity of the test function $\rho$ we infer from the previous proposition for $\tilde N=0$  that
\begin{align*}
\begin{split}
|K_{ \tilde s_\mu}&(x,y) |\leq \sqrt{\sum_{j \geq 0} \rho(\mu-\mu_j) |\phi_j(x)|^2} \sqrt{\sum_{j \geq 0} \rho(\mu-\mu_j) |\phi_j(y)|^2}=\sqrt{K_{ \tilde s_\mu}(x,x)}  \sqrt{K_{ \tilde s_\mu}(y,y)}=O\big (\mu^{d-1}  \big )
\end{split}
\end{align*}
uniformly in $x,y\in M$. Also, note that the asymptotics in the proposition  off the diagonal are only meaningful if $\dist(x,y)^{-1}$ is small with respect to  $\mu$. 
\end{rem}

As the previous proposition shows, the kernel of $\tilde s_\mu$ exhibits a caustic behaviour\footnote{For the terminology, see Appendix A in \cite{ramacher16}.} in a neighbourhood of the diagonal since for $x=y$  the integrals $I_\iota(\mu,R,t,x,y)$ no longer oscillate and are of order $O(\mu^0)$.  From Proposition \ref{prop:31.05.2017}, similar asymptotics for  $s_\mu$ can be deduced. By looking at asymptotics on the diagonal, one obtains Weyl's law for the spectral function of $Q$ and convex $\L^\infty$-bounds for eigenfunctions, since $\norm {s_\mu}_{\L^2 \to \L^\infty}^2\equiv \sup_{x\in M} s_\mu(x,x)$, yielding  for any eigenfunction the convex bound
\bqn
\norm{\phi_j}_{\infty} \ll \, \lambda_j^{\frac{d-1}{2m}},
\eqn
compare  \cite[Theorem 5.1]{hoermander68} and \cite[Eq. (3.2.6)]{sogge14}. Nevertheless,  in order to prove subconvex bounds, we shall also need asymptotics off the diagonal, so that the full caustic behaviour of $K_{\tilde s_\mu}(x,y)$ near the diagonal becomes relevant.

\subsection{The reduced spectral function and equivariant convex bounds for eigenfunctions}  Keeping the notation of Section \ref{sec:spec.funct}, assume now  that $M$ carries an effective and isometric action of a compact  Lie group $K$, and consider  the right regular representation $\pi$ of $K$ on $\L^2(M)$ with corresponding 
 Peter-Weyl decomposition
\bq
\label{eq:PW}
\L^2(M) = \bigoplus_{\sigma \in \widehat K} \L^2_\sigma(M), \qquad \L^2_\sigma(M):= \Pi_\sigma \L^2(M),
\eq
where $\widehat K$ denotes the unitary dual of $K$ and 
\bqn 
\Pi_\sigma:= d_\sigma \int_K \overline{\sigma(k)} \pi(k) \d k
\eqn
 the orthogonal projector onto the $\sigma$-isotypic component, $dk$ being Haar measure and $d_\sigma$ the dimension of an irreducible representation of $K$ in the class $\sigma \in \widehat K$. Further, suppose that $P$ commutes with $\pi$, and that the orthonormal basis $\mklm{\phi_j}_{j \geq 0}$ is compatible with the decomposition \eqref{eq:PW} in the sense that each $\phi_j$ lies in some $\L^2_\sigma(M)$. Then every eigenspace of $P$ is invariant under $\pi$, and decomposes into irreducible $K$-modules spanned by  eigenfunctions. In order to study eigenfunctions of $P$ of a certain $K$-type, one is interested in the spectral function of the operator $Q_\sigma:=\Pi_\sigma \circ Q \circ \Pi_\sigma= \Pi_\sigma \circ Q= Q \circ \Pi_\sigma$, also called the \emph{reduced spectral function}, given by 
 \bq
 \label{eq:02.02.2017} 
 e_\sigma(x,y,\mu)=\sum_{\mu_j \leq \mu, \, \phi_j \in \L^2_\sigma(M)} \phi_j(x) \overline{ \phi_j(y)}. 
 \eq
 For this, one considers the composition $s_\mu \circ \Pi_\sigma$, or rather $
\tilde s_\mu \circ \Pi_\sigma$, whose kernel has the spectral expansion
\bq
\label{eq:24.11.2016b}
K_{\widetilde s_\mu \circ \Pi_\sigma}(x,y)=\sum_{j \geq 0,  \phi_j \in \L^2_\sigma(M)} \rho(\mu-\mu_j) \phi_j(x) \overline{\phi_j(y)}.
\eq
Similarly to \eqref{eq:20.01.2017bis}, it was shown in \cite[Eq. (2.8)]{ramacher16} that by approximating $U(t)$ in \eqref{eq:19.1.2017} by the Fourier integral operator $\bar U(t)$ one obtains a  description for the kernel of $\widetilde s_\mu \circ \Pi_\sigma$ as the double oscillatory integral
\begin{align}
\label{eq:20.01.2017}
\begin{split}
K_{\widetilde s_\mu \circ \Pi_\sigma}(x,y)
=&  \frac {\mu^d d_\sigma  }{(2\pi)^{d+1}}  \sum_\iota  \int_{\R}\int _{\R} e^{i\mu[t-Rt]}   I^\sigma_{\iota}(\mu, R, t, x,y)  \d R\d t 
\end{split}
\end{align}
up to terms of order  $O(\mu^{-\infty})$ which are uniform in $x$ and $y$, where
\begin{align*}
\begin{split}
 I^\sigma_{\iota}(\mu, R, t, x,y):= &    \int_K \int_{\Sigma^{R,t}_{\iota,x}} e^{i{ \mu}  \Phi_{\iota,x,y}(\omega,k)} \hat \rho(t)  \overline{\sigma(k)}  f_\iota( x) \\ &\cdot     a_\iota(t, \kappa_\iota(  x) , \mu \omega)  \bar f _\iota (y \cdot k^{-1})   b(q( x, \omega)) J_\iota(k,y) {\d\Sigma^{R,t}_{\iota,x}(\omega)  } \d k,  \\
 \Phi_{\iota,x,y}(\omega,k):=&\eklm{\kappa_\iota(  x) - \kappa_\iota(y \cdot  k^{-1}),\omega},
 \end{split}
\end{align*}
 $J_\iota(k,y)$ being a Jacobian.   Write $\O_x:= x \cdot K$ for the $K$-orbit through $x \in M$. We then have the following 
  
\begin{proposition}
\label{prop:30.01.2017}
Suppose that $K$ acts on $M$ with orbits of the same dimension $\kappa\leq d-1$ and  that the cospheres $S^\ast_xM:=\mklm{(x,\xi) \in T^\ast M\mid  p(x,\xi)=1}$ are strictly convex. Then, for any  fixed $x, y \in M$, $\sigma \in \widehat K$,  and $\tilde N=0,1,2,3,\dots$ one has the expansion
\begin{align*}
\begin{split}
K_{\widetilde s_\mu \circ \Pi_\sigma}(x,y)&=    \mu^{d-{\frac{\epsilon_{x,y}}{2}}-1} d_\sigma  \left [ \sum_{r=0}^{\tilde N-1} \Lcal_r^\sigma (x,y,\mu) + \cR^\sigma_{\tilde N}(x,y,\mu) \right ]
\end{split}
\end{align*}
up to terms of order $O(\mu^{-\infty})$ as $\mu \to +\infty$, where  
\bqn 
\varepsilon_{ x,y}:=\begin{cases} 2 \kappa, & y \in \mathcal{O}_x, \\ d-1+\kappa, & y \notin \mathcal{O}_x.  \end{cases}
\eqn
The coefficients in the expansion and  the remainder term can be computed explicitly;  if $y \in \O_x$, they  satisfy the bounds
 \begin{align*}
\begin{split}
\Lcal^\sigma_r(x,y,\mu ) &\ll \, \sup_{u \leq 2r}\norm{D^u \sigma}_\infty\, \mu^{-r}, \qquad
 \cR^\sigma_{\tilde N}(x,y,\mu) \ll \, \sup_{u \leq 2\tilde N+ \lfloor \frac \kappa 2+1\rfloor}\norm{D^u \sigma}_\infty \, \mu^{-\tilde N},
\end{split}
\end{align*}
uniformly in $x$ and $y$, where $D^u$ denote differential operators on $K$ of order $u$, and  if $y \notin \O_x$, the bounds
 \begin{align*}
\begin{split}
\Lcal^\sigma_r(x,y,\mu ) &\ll \, \sup_{u \leq 2r}\norm{D^u \sigma}_\infty \cdot {\dist (x, \O_y)}^{-\frac{d-\kappa-1}2-r}\, \mu^{-r}, \\ 
 \cR^\sigma_{\tilde N}(x,y,\mu) &\ll \, \sup_{u \leq 2\tilde N + \lfloor\frac{\kappa}2+1\rfloor}\norm{D^u \sigma}_\infty \cdot \dist (x, \O_y)^{-\frac{d-\kappa-1}2-\tilde N} \, \mu^{-\tilde N},
\end{split}
\end{align*}
where $\dist (x, \O_y):=\min\mklm{\dist(x,z)\mid z \in \O_y}$. If $K=T$ is a torus,  let $\widehat T' \subset \widehat T$ be the subset of representations occuring in the decomposition \eqref{eq:PW}, and identify  $\widehat T$ with the set of integral linear forms on $\t$. Then the remainder estimates can be improved to
\bqn 
 \cR^\sigma_{\tilde N}(x,y,\mu) \ll \, \sup_{u \leq 2\tilde N}\norm{D^u \sigma}_\infty \, \mu^{-\tilde N}, \qquad  \cR^\sigma_{\tilde N}(x,y,\mu) \ll \, \sup_{u \leq 2\tilde N}\norm{D^u \sigma}_\infty \cdot \dist (x, \O_y)^{-\frac{d-\kappa-1}2-\tilde N} \, \mu^{-\tilde N},
\eqn
respectively, provided that $\sigma \in \mathcal{V}_\mu:=\{\sigma' \in \widehat T' \mid |\sigma'| \leq C \mu/\log \mu\}$ for some constant $C>0$.

\end{proposition}  

\begin{rem}
\label{prop:01.02.2017}
\hspace{0cm}
\begin{enumerate}
\item  
Proposition \ref{prop:30.01.2017} implies for $\tilde N=0$  by Cauchy-Schwarz that
\begin{align*}
\begin{split}
|K_{ \tilde s_\mu \circ \Pi_\sigma}&(x,y) |\leq \sqrt{\sum_{j \geq 0,  \phi_j \in \L^2_\sigma(M)} \rho(\mu-\mu_j) |\phi_j(x)|^2} \sqrt{\sum_{j \geq 0,  \phi_j \in \L^2_\sigma(M)} \rho(\mu-\mu_j) |\phi_j(y)|^2}\\ 
&=\sqrt{K_{ \tilde s_\mu \circ \Pi_\sigma}(x,x)}  \sqrt{K_{ \tilde s_\mu \circ \Pi_\sigma}(y,y)}=O \Big (d_\sigma  \mu^{d-\kappa-1} \sup_{u \leq  \lfloor \kappa/2+1\rfloor }\norm{D^u \sigma}_\infty \Big)
\end{split}
\end{align*}
uniformly in $x,y\in M$ and $\sigma \in \widehat K$, while taking $\tilde N=1$ would yield an estimate of order 
$$O\Big (d_\sigma \,  \mu^{d-\kappa-1} \big ( \norm{\sigma}_\infty  +  \sup_{u \leq  \lfloor \kappa/2+3\rfloor }\norm{D^u \sigma}_\infty \mu^{-1}\big ) \Big ).$$
 If $K=T$ is a torus, better remainder estimates hold. 
\item Note that the asymptotics of Proposition \ref{prop:30.01.2017} in  the case $y \not \in \O_x$ are only meaningful if $\dist(x, \O_y)^{-1}$ is small with respect to  $\mu$. 
\end{enumerate}
\end{rem}

\begin{proof}[Proof of Proposition \ref{prop:30.01.2017}]
{This proposition is essentially a consequence of \cite[Theorem 3.3]{ramacher16} and \cite[Theorem 3.2]{ramacher18}. In particular,  an asymptotic expansion of $K_{\widetilde s_\mu \circ \Pi_\sigma}(x,x)$ was given in \cite[Proposition 4.2]{ramacher16} and, with an improved remainder estimate in the toric case, in  \cite[Proposition 4.1]{ramacher16}. } To obtain an asymptotic expansion off  the diagonal from \eqref{eq:20.01.2017}, we shall first apply the stationary phase theorem  to the integrals $ I^\sigma_{\iota}(\mu, R, t, x,y)$, and then to the $(R,t)$-integral. 
 If $x \not\in Y_\iota$ or $\O_y\cap Y_\iota=\emptyset$, $I^\sigma_\iota(\mu,R,t,x,y)=0$.  Otherwise,  \cite[Theorem 3.3]{ramacher16} implies for sufficiently small $Y_\iota$,  fixed $R,t \in \R$, and  any $\tilde N\in \N$ the asymptotic expansion
\bq
\label{eq:13.06.2016b}
 I^\sigma _{\iota}(\mu, R, t, x,y)= \mu^{-\frac{\mathrm{codim}\, \Crit \, \Phi_{\iota, x,y}}{2}}  e^{i\mu \Phi_{\iota, x,y}^0(R,t)}  \left [ \sum_{r=0}^{\tilde N-1} \Lcal^\sigma_{\iota,r}(R,t,x,y) \mu^{-r}+ \mathcal{\tilde R}^\sigma_{\iota,\tilde N}( R,t,x,y,\mu)\right ],
\eq
where $\Crit \, \Phi_{\iota, x,y}$ denotes the critical set of $\Phi_{\iota,x,y}$, and 
\bqn 
\mathrm{codim}\,  \Crit \, \Phi_{\iota, x,y}=\begin{cases} 2 \kappa, & y \in \mathcal{O}_x, \\ d-1+\kappa, & y \notin \mathcal{O}_x. \end{cases}
\eqn
 The coefficients   and the remainder term are given by distributions depending smoothly on $R,t$ with support in $\Crit  \, \Phi_{\iota,x,y}$ and $\Sigma^{R,t}_{\iota,x} \times K$, respectively.  Furthermore, they and their derivatives with respect to $R,t$ satisfy  for $y \not\in \O_x$ the bounds
 \begin{align*}
\begin{split}
\gd ^\alpha_{R,t}\Lcal^\sigma_{\iota,r}(R,t,x,y) &\ll \sup_{u \leq 2r}\norm{D^u \sigma}_\infty \cdot   {\dist (x, \O_y)}^{-\frac{d-\kappa-1}2-r}, \\ 
\gd^\alpha_{R,t}\mathcal{\tilde R}^\sigma_{\iota,\tilde N}( R,t,x,y,\mu)  &\ll \sup_{u \leq 2\tilde N + \lfloor\frac{\kappa}2+1\rfloor}\norm{D^u \sigma}_\infty \cdot  \dist (x, \O_y)^{-\frac{d-\kappa-1}2-\tilde N} \, \mu^{-\tilde N}
\end{split}
\end{align*}
while for $y \in \O_x$ one has
 \begin{align*}
\begin{split}
\gd^\alpha_{R,t} \Lcal^\sigma_{\iota,r}(R,t,x,y) &\ll \sup_{u \leq 2r}\norm{D^u \sigma}_\infty, \qquad \gd^\alpha_{R,t} \mathcal{\tilde R}^\sigma _{\iota, \tilde N}( R,t,x,y,\mu)   \ll \sup_{u \leq 2\tilde N+ \lfloor \frac\kappa 2+1\rfloor}\norm{D^u \sigma}_\infty \mu^{-\tilde N} 
\end{split}
\end{align*}
uniformly in $x$ and $y$.  {If $K=T$ is a torus and $\sigma \in \mathcal{V}_\mu$,  the remainder estimates can be improved   \cite[Theorem 3.2]{ramacher18} to contain only derivatives of $\sigma$ up to order $2\tilde N$.}
 Finally,
\bqn
 \Phi_{\iota,x,y}^0(R,t)= R \,  c_{x,y} (t), \qquad c_{x,y} (t)= \pm \frac{\norm{\kappa_\iota(x)-\kappa_\iota(y\cdot k_0^{-1})}}{\grad_\eta \zeta_\iota(t, \kappa_\iota(x), \omega_0)},
 \eqn
  denotes the constant value(s) of $\Phi_{\iota,x,y}$ on (the components of) its critical set, where $(\omega_0,k_0)$  is some point in $\Crit \, \Phi_{\iota,x,y}$.  If $y \in \O_x$ one has $\Phi_{\iota,x,y}^0(R,t)=0$. As already noted in the proof of Proposition \ref{prop:31.05.2017},  $a_\iota$ is a polyhomogeneous symbol of order $0$, so that  the  dependence of the amplitude on $\mu$ does not interfer with the asymptotics.  Putting \eqref{eq:20.01.2017} and \eqref{eq:13.06.2016b} together we obtain 
\begin{align*}
\begin{split}
K_{\widetilde s_\mu \circ \Pi_\sigma}(x,y) =&   \mu^{d-\frac{\mathrm{codim}\, \Crit \, \Phi_{\iota, x,y}}{2}}    d_\sigma \sum_\iota   \int_{\R}\int _{\R} e^{i\mu[t-Rt]} e^{i\mu \Phi_{\iota, x,y}^0(R,t)}  \left [  \sum_{r=0}^{\tilde N-1} \Lcal^\sigma_{\iota,r}(R,t,x,y) \mu^{-j} \right. \\ 
&\left. + \mathcal{\tilde R}^\sigma _{\iota,\tilde N}( R,t,x,y,\mu)    \right ]  \d R\d t\\
\end{split}
\end{align*}
up to terms of order $O(\mu^{-\infty})$ uniform in $x$ and $y$. We now apply the stationary phase principle \cite[Theorem 7.7.5]{hoermanderI} to the $(R,t)$-integral. If $y \in \O_x$,  the phase function  simply reads $t(1-R)$, and the only critical point is $(R_0,t_0)=(1,0)$, which is non-degenerate, the determinant of  the Hessian being $-1$. Therefore, the necessary conditions for an application of the principle are fulfilled, yielding the assertion of the proposition in this case.  In case that $y \not \in \O_x$, the phase function is given by $t(1-R) +\Phi^0_{\iota,x,y}(R,t)$, and the determinant of the matrix of its second derivatives is given by
\bqn
-(1-c'_{x,y}(t))^2 \approx- \left (1  \pm O({\norm {\kappa_\iota(x) - \kappa_\iota(y \cdot k_0^{-1} )}})    \right )^2.
\eqn
  By choosing the charts $Y_\iota$ sufficiently small so that $\|\kappa_\iota(x) - \kappa_\iota(y \cdot k_0^{-1})\| \ll 1$, we can  achieve that  in a  sufficiently small neighborhood of $(R,t)=(1,0)$ the phase function $t(1-R) +\Phi^0_{\iota,x,y}(R,t)$ has, if at all, only non-degenerate, hence isolated, critical points.  If we now apply the stationary phase theorem, the proposition follows. 
\end{proof}

From Proposition \ref{prop:30.01.2017}  equivariant convex bounds for eigenfunctions  can be easily inferred. Indeed, recall that the test function $\rho \in \S(\R,\R^+)$ was chosen such that $\hat \rho(0)=1$ and $\supp \hat \rho \subset (-\delta/2,\delta/2)$ for some arbitrary $\delta>0$.  By choosing $\delta$ sufficiently small, one can even achieve that $\rho >0$ on $[-1,1]$, compare  \cite[Proof of Lemma 2.3]{duistermaat-guillemin75}. But then 
\begin{align*}
\begin{split}
 \min_{ \nu \in [-1,1]}\rho(\nu)\underbrace{\sum_{{\mu_j \in (\mu,\mu+1],}{\, \phi_j \in \L^2_\sigma(M)}} |\phi_j(x)|^2}_{=K_{s_\mu \circ \Pi_\sigma}(x,x)}   & \leq  \underbrace{\sum_{j \geq 0, \, \phi_j \in \L^2_\sigma(M)} \rho(\mu-\mu_j) |\phi_j(x)|^2}_{= K_{ \tilde s_\mu \circ \Pi_\sigma}(x,x)}\\ & =O\Big (d_\sigma \,  \mu^{d-\kappa-1}   \sup_{u \leq \lfloor \kappa/2+1\rfloor}\norm{D^u \sigma}_\infty  \Big ),
\end{split}
\end{align*}
yielding a corresponding bound for  $K_{s_\mu \circ \Pi_\sigma}(x,x)$,  compare \cite[Remark 4.4 (2)]{ramacher16}. In view of the equality $\norm {s_\mu\circ \Pi_\sigma}_{\L^2 \to \L^\infty}^2\equiv \sup_{x\in M} K_{ s_\mu \circ \Pi_\sigma}(x,x)$, one finally obtains the   equivariant convex bound
\bqn
\norm{\phi_j}_{\infty} \ll \, \lambda_j^{\frac{d-\kappa-1}{2m}} \left (\sqrt{  d_\sigma  \sup_{u \leq  \lfloor \kappa/2+1\rfloor}\norm{D^u \sigma}_\infty }\right )
\eqn
for any  $\phi_j\in \L^2_\sigma(M)$ and $\sigma \in \widehat K$, see  \cite[Proposition 5.1 and Eq. (5.4)]{ramacher16}. As in the non-equivariant case, the kernels $K_{\tilde s_\mu \circ \Pi_\sigma}(x,y)$ exhibit a caustic behaviour in their dependence on $x,y$, which will be crucial for the derivation of equivariant subconvex bounds. {In case that $K=T$ is a torus and $\sigma \in \mathcal{V}_\mu$, the bounds above are independent of $\sigma$, see \cite[Proposition 5.1]{ramacher18}.}

\section{Spectral asymptotics for kernels of Hecke operators}
\label{sec:4}

Keep  the notation of  Sections \ref{sec:hecke} and \ref{sec:FIO}.  The main goal of this paper consists in proving subconvex bounds for {Hecke--Maass forms of rank $1$} on the compact $d$-dimensional Riemannian manifold $M=\Gamma_\chi \bsl G$. To this purpose, we shall first derive asymptotics for kernels of Hecke operators in the eigenvalue and isotypic aspect. 
Recall  that $K$ acts on $G$ and  $M$ from the right  in an isometric  and effective way,  the isotropy group of a point $\Gamma_\chi g\in \Gamma_\chi \bsl G$ being conjugate 
to the finite group $gKg^{-1}\cap \Gamma_\chi$. Hence, all $K$-orbits in $\Gamma_\chi\bsl G$ are either principal or exceptional, and of dimension $\dim K$. Since the maximal compact subgroups of $G$ are precisely the conjugates of $K$, exceptional $K$-orbits arise from elements in $\Gamma_\chi $ of finite order.  
Consider now the right regular representation $\pi$ of $K$ on $\L^2(\Gamma_\chi\bsl G)$ together with the corresponding  Peter-Weyl decomposition \eqref{eq:PW}, and suppose that  $P$ commutes with $\pi$ and  the Hecke operators $\T^\chi_{\Gamma \beta \Gamma}$, which commute with the right regular $K$-representation as well.
To describe  the growth of simultaneous eigenfunctions of $P$ and $T^\chi_{\Gamma \beta \Gamma}$ in  the $\sigma$-isotypic component 
\bq
\label{eq:4.1}
\L^2_{\sigma,\chi}(\Gamma\bsl G):=\L^2_\sigma(\Gamma_\chi\bsl G)\cap \L^2_\chi(\Gamma\bsl G), \qquad \chi \in \widehat \Gamma, \, \sigma \in \widehat K, 
\eq
of $\L^2_\chi(\Gamma\bsl G)$, we are interested in spectral asymptotics for the Schwartz kernel of  the operator 
$$
\Pi_\sigma  \circ T_{h(\chi)} \circ T^\chi_{\Gamma \beta \Gamma} \circ Q   \circ T_{h(\chi)}  \circ \Pi_\sigma = \T^\chi_{\Gamma \beta \Gamma}\circ Q    \circ \Pi_\sigma: \, \L^2 (\Gamma_\chi\bsl G) \, \longrightarrow \, \L^2 (\Gamma_\chi\bsl G).
$$
Let $\mklm{\phi_j}_{j\geq 0}$ be an orthonormal basis of $\L^2(\Gamma_\chi \bsl G)$  consisting of simultaneous eigenfunctions  of $P$ and $\T^\chi_{\Gamma \beta \Gamma}$ compatible with the decompositions \eqref{eq:directsum} and  \eqref{eq:PW}. Applying the Hecke operators $\T^\chi_{\Gamma \beta \Gamma}$ to the spectral expansion \eqref{eq:02.02.2017} of the spectral function of $Q \circ \Pi_\sigma$ yields  
 \bq
\label{eq:pretrace}
\sum_{\mu_j \leq \mu, \, \phi_j \in \L^2_\sigma(M)}  \lambda_j(\chi, \beta) \phi_j(x) \overline{ \phi_j(y)}=[\Gamma:\Gamma_\chi]^{-1} \sum_{\stackrel{\mu_{j}\leq \mu, \, \phi_j \in \L^2_\sigma(\Gamma_\chi\bsl G)}{\alpha \in \Gamma_\chi \bsl\Gamma\beta\Gamma}} \overline{\chi(\alpha)} \phi_j(\alpha \cdot x) \overline{\phi_j(y)}.
\eq
  In order to get an asymptotic description of the right-hand side of \eqref{eq:pretrace}, we consider    the composition  $\T^\chi_{\Gamma \beta \Gamma} \circ \tilde s_\mu \circ \Pi_\sigma$ with the approximate spectral projection $\tilde s_\mu$.  Clearly, its Schwartz kernel can be written as  
\bq
\label{eq:24.11.2016a}
K_{\T^\chi_{\Gamma \beta \Gamma} \circ \tilde s_\mu \circ \Pi_\sigma}(x,y) :=\frac 1{[\Gamma:\Gamma_\chi]} \sum_{\alpha \in \Gamma_\chi \bsl \Gamma \beta \Gamma} \overline{\chi(\alpha)} K_{\tilde s_\mu \circ \Pi_\sigma}(\alpha \cdot x,y),
\eq
$K_{\tilde s_\mu \circ \Pi_\sigma}( x,y)$ being as in \eqref{eq:24.11.2016b}, and by Remark \ref{prop:01.02.2017} (1) one immediately deduces 
\bqn
K_{\T^\chi_{\Gamma \beta \Gamma} \circ \tilde s_\mu \circ \Pi_\sigma}(x,x)\ll   \frac{|\Gamma_\chi \bsl \Gamma \beta \Gamma|}{[\Gamma:\Gamma_\chi]} \, d_\sigma \,  \mu^{d-\dim K-1}   \sup_{u \leq \lfloor \dim K/2+1\rfloor}\norm{D^u \sigma}_\infty
\eqn
uniformly in $x$, { or if $K=T$ is a torus, 
\bqn
K_{\T^\chi_{\Gamma \beta \Gamma} \circ \tilde s_\mu \circ \Pi_\sigma}(x,x)\ll   \frac{|\Gamma_\chi \bsl \Gamma \beta \Gamma|}{[\Gamma:\Gamma_\chi]}  \,  \mu^{d-\dim K-1}, \qquad \sigma \in \mathcal{V}_\mu.
\eqn
} 
Nevertheless, to obtain subconvex bounds, more subtle estimates are necessary that take into account the caustic behaviour of the kernels $K_{\tilde s_\mu \circ \Pi_\sigma}( x,y)$ near the diagonal. 
  
\begin{proposition} 
\label{thm:15.09.2016}
Let $\chi \in \widehat \Gamma$, $\sigma \in \widehat K$, $x  \in \Gamma_\chi \bsl G$,  and $\tilde N\in \N$ be arbitrary. Assume that  the cospheres $S^\ast_x(\Gamma_\chi \bsl G)$ are strictly convex. Then, for any $0 <\delta \ll 1$,
\[
K_{\T^\chi_{\Gamma \beta \Gamma} \circ \tilde s_\mu \circ \Pi_\sigma}(x,x)\ll  \frac{d_\sigma}{[\Gamma:\Gamma_\chi]}  \sup_{u \leq  \lfloor \frac{\dim K}2+1\rfloor }\norm{D^u \sigma}_\infty \left\{  \mu^{d-\dim K-1}  M(x,\beta,\delta) +  \mu^{\frac{d-\dim K -1}2} \int_\delta^{C} s^{-\frac12} dM(s)  \right\}
\]
uniformly in $x \in \Gamma_\chi \bsl G$ for some sufficiently large\footnote{It suffices to take $C$ larger than the diameter of any component of $\Gamma_\chi \bsl G/K$.} constant $C>0$ up to terms of order $O(|\Gamma_\chi \bsl \Gamma \beta \Gamma| \;\mu^{-\infty})$, where we set 
\begin{align}
\label{eq:12.02.2017}
M(\delta):=M(x,\beta,\delta)&:=\# \mklm{\alpha \in \Gamma_\chi\bsl\Gamma\beta\Gamma: \, \dist(xK,\alpha \cdot xK)^{d-\dim K- 1} <\delta }
\end{align}
and $\dist(xK,\alpha \cdot xK)\equiv \dist(\Gamma xK,\Gamma \alpha  xK)$. {If $K=T$ is a torus and $\sigma \in \mathcal{V}_\mu$, the better estimate
\[
K_{\T^\chi_{\Gamma \beta \Gamma} \circ \tilde s_\mu \circ \Pi_\sigma}(x,x)\ll  \frac{1}{[\Gamma:\Gamma_\chi]}   \left\{  \mu^{d-\dim K-1}  M(x,\beta,\delta) +  \mu^{\frac{d-\dim K -1}2} \int_\delta^{C} s^{-\frac12} dM(s)  \right\}
\]
holds. }
\end{proposition}
\begin{proof}
By   Proposition \ref{prop:30.01.2017} one deduces as $\mu\to +\infty$
\begin{gather*}
 \sum_{\stackrel{\alpha \in \Gamma_\chi \bsl \Gamma \beta \Gamma,}{\dist (xK,\alpha \cdot xK)^{d-\dim K-1} \geq \delta}} \overline{\chi(\alpha)} K_{\tilde s_\mu \circ \Pi_\sigma}(\alpha \cdot x,x) \\ \nonumber \ll d_\sigma \, \mu^{(d-\dim K-1)/2}\sum_{\stackrel{\alpha \in \Gamma_\chi \bsl \Gamma \beta \Gamma,} {\dist (xK,\alpha \cdot xK)^{d-\dim K-1} \geq \delta}}   \dist(xK,\alpha \cdot xK)^{-(d-\dim K-1)/2}\sup_{u \leq \lfloor \frac{\dim K}2 +1\rfloor} \norm{D^u \sigma}_\infty.
\end{gather*}
Furthermore, by Remark \ref{prop:01.02.2017} (1) one has the uniform bound 
$$
K_{ \widetilde s_\mu \circ \Pi_\sigma}(x,y) =O\Big (d_\sigma \,  \mu^{d-\dim K-1}  \sup_{u \leq  \lfloor \frac{\dim K}2 +1\rfloor}\norm{D^u \sigma}_\infty  \Big ).$$ 
In view of  \eqref{eq:24.11.2016a} we therefore obtain 
\begin{align*}
{[\Gamma:\Gamma_\chi]} &  K_{\T^\chi_{\Gamma \beta \Gamma} \circ \tilde s_\mu \circ \Pi_\sigma}(x,x) = \sum_{\stackrel{\alpha \in \Gamma_\chi \bsl \Gamma \beta \Gamma,}{\dist (xK,\alpha \cdot xK)^{d-\dim K-1} <\delta}}  \overline{\chi(\alpha)} K_{\tilde s_\mu \circ \Pi_\sigma}(\alpha \cdot x,x)\\ &+\sum_{\stackrel{\alpha \in \Gamma_\chi  \bsl \Gamma \beta \Gamma,}{\dist (xK,\alpha  \cdot xK)^{d-\dim K-1} \geq \delta}} \overline{\chi(\alpha)} K_{\tilde s_\mu \circ \Pi_\sigma}(\alpha \cdot x,x)\\
&\ll  d_\sigma \,  \mu^{d-\dim K-1}  \sup_{u \leq \lfloor \frac{\dim K}2 +1\rfloor}\norm{D^u \sigma}_\infty M(x,\beta,\delta) \\ &+  d_\sigma \, \mu^{(d-\dim K -1)/2}  \sup_{u \leq \lfloor \frac{\dim K}2 +1\rfloor} \norm{D^u \sigma}_\infty    \int_\delta^{C} s^{-1/2} dM(s) \end{align*}
by definition of the Stieltjies integral. {In case that $K=T$ is a torus, corresponding better estimates hold,  and the assertion follows.} 
\end{proof}

\section{Subconvex bounds on $\Gamma_\chi \bsl \SL(2,\R)$ for arithmetic congruence lattices}
\label{sec:IS}

In this section, we shall use the kernel asymptotics derived in the previous section  to prove  subconvex bounds on the quotient $\Gamma \bsl \SL(2,\R)$, where $\Gamma$ is an   arithmetic congruence lattice as considered by Iwaniec and Sarnak \cite{iwaniec-sarnak95}.

\subsection{Arithmetic congruence lattices} \label{sec:Arithmetic congruence lattices}
   To introduce the setting, let $A$ be an indefinite quaternion division algebra over $\Q$.
Hence, there exist two square-free integers $a$ and $b$ such that $a>0$ and
\[
A=\Q+\Q\, \omega+\Q\, \Omega+\Q\, \omega\, \Omega
\]
where $\omega^2=a$, $\Omega^2=b$, and $\omega\Omega=-\Omega\omega$.
For each element $x=x_0+x_1\omega+x_2\Omega+x_3\omega\Omega$, its conjugate  is defined as $\overline{x}:=x_0-x_1\omega-x_2\Omega-x_3\omega\Omega$, and  its trace and norm as $\tr(x):=x+\overline{x}$ and $N(x):=x\overline{x}$, respectively. 
Let $\cR$ be an \emph{order} of $A$, that is, $\cR$ is a finitely generated free $\Z$-module, $\cR$ is a subring of $A$ cotaining $1$, and $\cR\otimes_\Z\Q=A$.
For each prime number $p$,  set $A_p:=A\otimes\Q_p$ and $\cR_p:=\cR\otimes\Z_p$.
Let $d_A$ be the product of all primes $p$ such that $A_p$ is a division algebra. Then $d_A$ is called the \emph{discriminant} of $A$. $d_A$ is greater than $1$ and square free, and $A_p$ is isomorphic to $M(2,\Q_p)$ if $p$ does not divide $q$.
Throughout this section, we assume that $\cR$ is an \emph{Eichler order of level $L$}, where $L$ is a natural number such that $(d_A,L)=1$.
Hence, $\cR$ satisfies
\begin{enumerate}
\item $\cR_p$ is the maximal order of $A_p$ if $p$ divides $d_A$, or
\item $\cR_p$ is conjugate to $\begin{pmatrix}\Z_p&\Z_p \\ L\Z_p&\Z_p \end{pmatrix}$.
\end{enumerate}
Note that any Eichler order is included in a maximal order.
Particularly, $\cR$ is maximal when $L=1$. Now, choose an embedding $\theta:A\to M(2,\Q(\sqrt{a}))\subset M(2,\R)$ by setting 
\[
\theta(x_0+x_1\omega+x_2\Omega+x_3\omega\Omega)=\begin{pmatrix}x_0-x_1\omega & x_2+x_3\omega \\ b(x_2-x_3\omega) & x_0+x_1\omega \end{pmatrix}.
\]
For each natural number $n\in \N_\ast$, we set
\[
\cR(n):=\{\alpha\in \cR \mid N(\alpha)=n\}.
\]
Then $\Gamma:=\theta(\cR(1))$ becomes a cocompact lattice of $G:=\SL(2,\R)$.
Note that $\tr(x)=\tr(\theta(x))$ and $N(x)=\det(\theta(x))$ hold for any $x$ in $A$. In what follows, we identify $A$ with $\theta(A)$.
Especially, we will often use $\Gamma$ instead of $\cR(1)$.

Next, let $\chi$ be a Dirichlet character on $(\Z/L\Z)^\times$. In view of  the product  isomorphism $(\Z/L\Z)^\times\cong \prod_{p|L} (\Z_p/L\Z_p)^\times$  given by the diagonal embedding $a\mapsto (a)_p$, a character $\chi_p$ can be defined on $(\Z_p/L\Z_p)^\times$  by restriction of $\chi$ to each factor. Set
\[
\Xi_\cR:=\{\alpha\in\cR\mid N(\alpha)>0, \;  \;  (N(\alpha),L)=1 \}\quad \text{and} \quad \cR_L:=\{(x_p)_{p|L} \mid x_p\in \cR_p ,\;\; N(x_p)\not\in p\Z_p \},
\]
and define a character $\chi_L$ on the semigroup $\cR_L$ by
$$\cR_L\ni \left (\begin{pmatrix}a_p&b_p \\ Lc_p&d_p \end{pmatrix}\right)_{p|L}\mapsto \prod_{p|L}\overline{\chi_p(a_p)}\in\C.$$
Composing $\chi_L$ and the diagonal embedding $\Xi_\cR\subset \cR_L$, we obtain a character $\chi$ on the sub-semigroup $\Xi_\cR$ of $A^\times$. By the inclusion $\Gamma\subset \Xi_\cR$, $\chi$ becomes a character on $\Gamma$ which is called a \emph{Nebentypus character}. Notice that there are only finitely many Nebentypus characters for each fixed Eichler order $\cR$. Now, because $\Gamma$ and $\alpha^{-1}\Gamma\alpha$ are commensurable \cite[Proposition 4.1]{PR},  an inclusion map $\psi:\cR(n)\to C(\Gamma)$ is given by
\[
\psi(\alpha)=n^{1/2} \alpha ,\qquad  \alpha\in \cR(n).
\]
Since the subset $\cR(n)$ is left and right $\Gamma$-invariant, and it is known that $\Gamma\bsl \cR(n)$ is finite \cite[Section 5.3]{Miyake}, we can introduce the \emph{Hecke operators} 
\[
(T^\chi_nf)(x):=\sum_{\alpha\in \Gamma\bsl \cR(n)} \overline{\chi(\alpha)}\, f(\alpha\cdot x), \qquad f\in \L^2_\chi(\Gamma \bsl G).
\]
Indeed, since $\psi(\cR(n))$ is given as a disjoint union of double cosets $\sqcup_j \Gamma \alpha_j \Gamma$, the operator $T^\chi_n$ coincides with the sum  $\sum_j T^\chi_{\Gamma \alpha_j \Gamma}$ of Hecke operators defined in \eqref{eq:heckechar}. 
In particular, we have also the Hecke operators 
\bqn 
\T^\chi_n: \L^2(\Gamma_\chi \bsl G) \quad \longrightarrow\quad  \L^2(\Gamma_\chi\bsl G), \quad (\T^\chi_n f) (x):= \frac 1 {[\Gamma:\Gamma_\chi]} \sum_{ \alpha \in \Gamma_\chi \bsl \cR(n)}  \overline{\chi(\alpha)} f(\alpha \cdot x),
\eqn
compare \eqref{eq:heckechar0} and  Diagram \ref{diagram}. Furthermore, set
\[
q:=d_AL.
\]
For natural numbers $n$ such that $(n,q)=1$, the $T^\chi_n$ are self-dual, commute with the Beltrami--Laplace operator $\Delta$ on $G$, and satisfy the composition rule \cite[Section 5.3]{Miyake}
\bq
\label{eq:08.09.2016a}
T^\chi_nT^\chi_m=\sum_{d|(m,n)}d \cdot  {\chi(d)} \, T^\chi _{nm/d^2}.
\eq

Next,  recall that the group $\GL(2,\R)^+:=\{ x\in M(2,\R) \mid  \det(x)>0\}$ acts  transitively on the upper half plane $\mathbb{H}:=\mklm{z \in \C \mid \Im z>0}$ by fractional transformations\footnote{ Note that the center $\begin{pmatrix}a&0 \\ 0& a \end{pmatrix}$, $a\in\R^\ast$, acts trivially on $\bH$.}
\[
g \cdot z=\frac{az+b}{cz+d}, \qquad z\in\bH, \quad g=\begin{pmatrix}a&b \\ c&d \end{pmatrix},
\]
by which   $\bH$ becomes isomorphic to  the homogeneous space $G/K$, where $K:=\SO(2)$, and we define
\bq
\label{eq:j(g,z)}
j(g,z):=(cz+d) (\det g)^{1/2}. 
\eq
 In what follows, we shall identify $\Gamma_\chi\backslash G/K\simeq \Gamma_\chi \bsl \bH$ with a subset in $\bH$, and endow it  with the standard hyperbolic distance on $\bH$ given by \cite[Section 1.1]{Iwaniec}
\bqn 
\dist_\bH (z,w) :=\mathrm{arcosh} \, \big (1+ u(z,w)/2\big )=\ln\left (\frac{|z-\overline w| +|z-w|} {|z-\overline w| -|z-w|}\right  )\approx |z-w|, 
\eqn
where 
\bq
\label{eq:hypdist}
u(z,w):=\frac{|z-w|^2}{\Im z \, \Im w}, \qquad z, w \in \mathbb{H}. 
\eq
By this,  $\Gamma_\chi \bsl \bH$  becomes a compact hyperbolic surface. Note that $\dist_\bH$ agrees with  the distance function $\dist$ introduced at the beginning of Section \ref{sec:hecke}. Furthermore, one has the following important result of Iwaniec and Sarnak. 
\begin{lem} 
\label{lem:05.09.2016} For arbitrary $\eps>0$ one has
\bqn
 \# \mklm{ \alpha \in \mathcal{R}(n): u(z, \alpha \cdot z) < \delta} \ll_\epsilon ( \delta + \delta^{1/4}) n^{1+\eps} + n^\eps
\eqn
uniformly in $z$.
\end{lem}
\begin{proof}
See   \cite[Lemma 1.3]{iwaniec-sarnak95}.
\end{proof}

\subsection{Equivariant subconvex bounds}\label{sec:Equivariant subconvex bounds} With the notation of the previous section, we shall first derive   subconvex bounds for Hecke--Maass forms in $\L_\chi^2(\Gamma \bsl G)$ in the eigenvalue and isotypic aspect for the Beltrami--Laplace operator $\Delta$. For this,  let $\mklm{\phi_j}_{j\geq 0}$ be an orthonormal basis of $\L^2(\Gamma_\chi \bsl G)$  consisting of simultaneous eigenfunctions  of $P_0=\Delta$ and $\T^\chi_n$ compatible with the decompositions \eqref{eq:directsum} and  \eqref{eq:PW}, where $X=G$ and $M=\Gamma_\chi \bsl G$, respectively,   so that with $(n,q)=1$
\begin{align}
\label{eq:08.09.2016b}
\begin{split}
\Delta\phi_j=\lambda_j \phi_j ,\qquad \, \, \,  \, \T^\chi_n\phi_j=\lambda_j(n)\phi_j.
\end{split}
\end{align}
Note that $\lambda_j(n)=0$ if $\phi_j \not \in \L^2_\chi(\Gamma \bsl G)$. Further, let $\sigma \in \widehat K$ be a fixed $K$-type, and $\L^2_{\sigma,\chi}(\Gamma \bsl G)$ be defined as in \eqref{eq:4.1}. When $\sigma$ is trivial, the space $\L^2_{\sigma,\chi}(\Gamma \bsl G)$ can be identified with $\L^2_\chi(\Gamma \bsl G/K)$.  In what follows, we shall make the  identification $\SO(2)\simeq S^1 \subset \C$,  so that the characters of $K$ are given by the exponentials 
\bqn
\sigma_l(e^{i\theta}):=e^{il\theta}, \qquad \theta \in [0,2\pi), \quad l \in \Z.
\eqn
Since all irreducible representations of $K$ are one-dimensional,  Proposition \ref{thm:15.09.2016}  yields for any $x \in \Gamma_\chi  \bsl G$ and {$\sigma_l$ with $|l| \ll \mu/\log \mu$ the estimate
\begin{gather*}
K_{\T^\chi_n \circ \tilde s_\mu \circ \Pi_{\sigma_l}}(x,x)\ll    \frac{1}{[\Gamma:\Gamma_\chi]} \left [ \mu \, M(x,n,\delta)  +   \mu^{\frac 12}  \int_\delta^{C} s^{-\frac 12} dM(s) \right ]  
\end{gather*}}
up to terms of order $O(|\Gamma_\chi \bsl \mathcal{R}(n)| \;\mu^{-\infty})$, 
where we set
\begin{align*}
M(\delta):=M(x,n,\delta)&:=\# \mklm{\alpha \in \Gamma_\chi\bsl\cR(n):  u (xK,\alpha \cdot xK) <\delta^2 },
\end{align*}
regarding $xK \in \Gamma_\chi  \bsl G/ K\simeq \Gamma_\chi  \bsl \bH$ as an element in $\C$, and took into account that 
 for suitable constants $c_1,c_2>0$
\bq
\label{eq:31.08.2016}
c_1 u(z,w) \leq \dist(z,w)^2\equiv \dist_\bH (z,w)^2 \leq c_2 u(z,w), \qquad z, w \in \Gamma_\chi \backslash \mathbb{H}. 
\eq
In order to derive a uniform bound for 
$
K_{\T^\chi_n \circ \widetilde s_\mu \circ \Pi_{\sigma_l} }(x,x),
$ 
note that by Lemma \ref{lem:05.09.2016} one has with $N(s):=s^{-1/2}$ and $\delta=\mu^{-1}$
\begin{align*}
 \int_{\delta}^C  s^{-1/2} d M(s)&=\underbrace{N(C) M(C)}_{\ll_\eps n^{1+\eps}}  -\underbrace{N(\delta)  M(\delta)}_{\ll_\eps \delta^{-1/2} [ ( \delta^2 + \delta ^{1/2}) n^{1+\eps}+n^\eps]} - \underbrace{\int_{\delta} ^C  M(s) { N'(s) \d s }}_{\ll_\eps (s^{3/2} + \log s) n^{1+\eps}+  s^{-1/2} n^\eps\vert^C_\delta}\\ 
 & \ll_\eps n^{1+\eps}+ \mu^{1/2} n^\eps + n^{1+\eps}\log\mu
 \ll_\eps (\mu^{1/2} + n \log\mu) n^\eps,
\end{align*}
 Taking everything together we have shown

\begin{thm}
\label{thm:23.11.2016}
For any $n \in \N_\ast$, $\mu>0$, {and $\sigma_l$ with $|l| \ll \mu /\log \mu$ the uniform bound
\bqn 
K_{\T^\chi_n \circ \widetilde s_\mu \circ \Pi_{\sigma_l} }(x,x) \ll_\eps \,  (\mu+ n \mu^{1/2}\log\mu) \, n^\eps
\eqn}
holds up to terms of order $O(n^{1+\eps} \,\mu^{-\infty})$, where $x\in \Gamma_\chi \bsl G$,  and  $\chi \in \widehat \Gamma$ is a Nebentypus character. 
\end{thm}
\qed 

\begin{rem}
The previous theorem is the non-spherical analogon of \cite[Lemma 1.2]{iwaniec-sarnak95}. Note that the bounds for the point pair invariants on $\bH$ used  by Iwaniec and Sarnak in order to show \cite[Lemma 1.2]{iwaniec-sarnak95} are better than ours by a factor $(1+u(\alpha\cdot z, z))^{-5/4}$ in the Stieltjes integral, but the lattice point counting function considered by them is unbounded, while ours is a priori bounded.
\end{rem}

 Following the original approach of Iwaniec and Sarnak, we shall now make use of arithmetic amplification to deduce from  Theorem \ref{thm:23.11.2016}    equivariant  subconvex bounds. Since we will later choose $n \ll \mu^A$ for some $A\in \N$, we can neglect the contributions of order $O(n^{1+\eps} \,\mu^{-\infty})$ in the following.  Thus, let $\chi \in \widehat \Gamma $,  $\sigma_l \in \widehat K$ be arbitrary, and $\{\phi_j\}_{j\in \N}$ as in \eqref{eq:08.09.2016b}. Writing
$
\eta_j(n):={\lambda_j(n)}/{\sqrt n} 
$
we deduce with  \eqref{eq:24.11.2016b},  \eqref{eq:24.11.2016a}, \eqref{eq:08.09.2016a}, and \eqref{eq:08.09.2016b} that\footnote{Since the $T^\chi_n$  are adjoint operators for all $n$ with $(n,q)=1$, one has $\eta_j(n)=\overline{\eta_j(n)}$, compare  \cite[Theorem 5.3.8]{Miyake}. More generally, $(T^\chi_{\Gamma \alpha\Gamma})^*=T^\chi_{\Gamma\alpha'\Gamma}$,  where $\alpha'=\det(\alpha)\alpha^{-1}$.} 
\begin{gather*}
\sum_{j \geq 0, \, \phi_j \in \L^2_{\sigma_l}(\Gamma_\chi \bsl G)} \rho(\mu-\mu_j) \phi_j(x) \overline{\phi_j(y)} \eta_j(m) \eta_j(n)
\\ = \sum_{d| (n,m)}  \sum_{j\geq 0, \, \phi_j \in \L^2_{\sigma_l}(\Gamma_\chi \bsl G)} \rho(\mu-\mu_j) \phi_j(x) \overline{\phi_j(y)} \chi(d) \eta_j\left(\frac{nm}{d^2}\right) = \sum_{d| (n,m)} \frac {d\chi(d)}{\sqrt{nm}} K_{\T^\chi_{nm/d^2}  \circ \widetilde s_\mu \circ \Pi_{\sigma_l} }(x,y).
\end{gather*}
{If  one replaces $\mu$ by $\mu \log \mu$ in Theorem \ref{thm:23.11.2016} one obtains for any $\sigma_l\in \widehat K$
\bqn 
K_{\T^\chi_n \circ \widetilde s_\mu \circ \Pi_{\sigma_l} }(x,x) \ll_\eps \,  \log \mu \,(\mu+ n \mu^{1/2}\log\mu) \, n^\eps,
\eqn}
yielding for arbitrary $N \in \N_\ast$ and $\sigma_l$
\begin{align*}
\sum_{j \geq 0, \, \phi_j \in \L^2_{\sigma_l}(\Gamma_\chi \bsl G)} \rho(\mu-\mu_j)  &|\phi_j(x)|^2 \Big |\sum_{n \leq N} z_n  \eta_j(n)\Big |^2 = \sum_{n,m \leq N} \sum_{d| (n,m)}    \frac {d \chi(d)}{\sqrt{nm}} \, z_n \overline{z_m} \, K_{\T^\chi _{nm/d^2} \circ \widetilde s_\mu \circ \Pi_{\sigma_l}}(x,x)\\
& \ll_\eps   N^\eps \mu^\eps \sum_{n,m \leq N} \sum_{d| (n,m)}    \frac d{\sqrt{nm}} |z_n z_m | \Big (\mu + \frac{nm}{d^2} \mu^{1/2} \Big ),
\end{align*}
since $|\chi(d)|=1$, where $z_n \in \C$ are arbitrary complex numbers. A simple computation then gives
\begin{align}
\begin{split}
\label{eq:14.02.2017}
\sum_{j \geq 0, \, \phi_j \in \L^2_{\sigma_l}(\Gamma_\chi \bsl G)} & \rho(\mu-\mu_j) |\phi_j(x)|^2 \Big |\sum_{n \leq N} z_n  \eta_j(n)\Big |^2  \ll_\eps  N^\eps \mu^\eps \left [\mu \sum_{n\leq N}   |z_n |^2 +N \mu^{1/2}   \Big  ( \sum_{n\leq N} |z_n| \Big )^2 \right ].
\end{split}
\end{align}
We thus arrive at 

\begin{proposition}
\label{thm:08.09.2016} For any $\mu>0$, ${\sigma_l} \in \widehat K$,  $\chi \in \widehat \Gamma$, and $N\in \N_\ast$ one has the estimate
\begin{align*}
\sum_{\stackrel{\mu \leq {\sqrt{\lambda_j}}\leq \mu+1,}{ \phi_j \in \L^2_{\sigma_l}(\Gamma_\chi \bsl G)}}  |\phi_j(x)|^2 \Big |\sum_{n \leq N} z_n  \eta_j(n)\Big |^2& \ll_\eps   N^\eps \mu^\eps \left [\mu \sum_{n\leq N}   |z_n |^2 +N \mu^{1/2}  \Big  ( \sum_{n\leq N} |z_n| \Big )^2 \right ]
\end{align*}
uniformly in $x \in \Gamma_\chi \bsl G$. 
\end{proposition}
\begin{proof} As explained at the end of Section \ref{sec:FIO}, the test function $\rho \in \S(\R,\R^+)$  can be chosen such that $\rho >0$ on $[-1,1]$. The proposition now follows from \eqref{eq:14.02.2017} and the estimate 
\begin{align*}
\sum_{\stackrel{\mu \leq {\sqrt{\lambda_j}}\leq \mu+1,}{ \phi_j \in \L^2_{\sigma_l}(\Gamma_\chi \bsl G)}}  |\phi_j(x)|^2 \Big |\sum_{n \leq N} z_n  \eta_j(n)\Big |^2  &\cdot \min \mklm{\rho(\mu) \mid \mu \in [-1,1]} \\ &\leq  
\sum_{j \geq 0, \, \phi_j \in \L^2_{\sigma_l}(\Gamma_\chi \bsl G)} \rho(\mu-\mu_j) |\phi_j(x)|^2 \Big |\sum_{n \leq N} z_n  \eta_j(n)\Big |^2.
\end{align*}
\end{proof}

\medskip 

Next, one proceeds as follows. Let $j_0\geq 0 $ be fixed such that $\phi_{j_0} \in \L^2_{{\sigma_l},\chi}(\Gamma \bsl G)$, and consider the \emph{amplifier}
\bq
\label{eq:27.5.2017a} 
z_n:= \begin{cases} \eta_{j_0}(p), &n=p \leq \sqrt N, \\ -1 & n=p^2 \leq N, \\ 0, & \text{otherwise}, \end{cases}
\eq
where $p$ is a prime not dividing $q$.
Note that  \eqref{diagram} and \eqref{eq:08.09.2016a} imply
\[
\eta_j(p)^2-\eta_j(p^2)=\begin{cases} 1 & \text{if $\phi_j \in \L^2_\chi(\Gamma \bsl G)$,} \\ 0 & \text{otherwise.} \end{cases}
\]
Hence, 
\bqn 
\Big |\sum_{n \leq N} z_n  \eta_{j_0}(n) \Big |= \sum_{p \leq \sqrt N, p \not \hspace{0.5mm}|\, q} 1= O\Big(\sqrt N/ \log N^{1/2}\Big), 
\eqn
by the Prime Number Theorem. Writing  $\lambda_j=1/4 +r_j^2$ and taking $\mu=r_{j_0}$ Proposition \ref{thm:08.09.2016} then gives 
\begin{align}
\label{eq:27.5.2017b}
|\phi_{j_0}(x)|^2 \ll_\eps   N^{\eps-1} r_{j_0}^\eps \left [ r_{j_0} \Big ( \sum_{p \leq \sqrt N} |\eta_{j_0}(p)|^2 + \sqrt N \Big ) + N r_{j_0}^{1/2} \Big ( \sum_{p\leq \sqrt N} |\eta_{j_0}(p)| +N^{1/2} \Big )^2 \right ].
\end{align}
As a next step, note that Jacquet-Langlands correspondence \cite{KR} and the study of Rankin-Selberg convolutions (\cite[Theorem 8.3]{iwaniec02} and \cite[Proposition 19.6]{Duke-Friedlander-Iwaniec}) imply  for any $j\in \N$ with $\phi_j \in \L^2_{\chi}(\Gamma \bsl G)$  the bound
\bq 
\label{eq:heckebound}
\sum_{n \leq N} |\eta_j(n)|^2 \ll_\eps N^{1+\eps} r_j^\eps, 
\eq
where $n$ moves over natural numbers prime to $q$. Here we used the facts that the Strong Multiplicity One Theorem holds for $\GL(2)$ and each automorphic representation factors as a tensor product of local representations.
Consequently, with \eqref{eq:heckebound} and Cauchy's inequality one deduces
\bqn 
|\phi_{j_0}(x)|^2 \ll_\eps    N^\eps r_{j_0}^\eps (r_{j_0} N^{-1/2}+ r_{j_0}^{1/2}N ).
\eqn
Choosing $N=r_{j_0}^{1/3}$ finally gives
\bqn 
|\phi_{j_0}(x)|^2 \ll_\eps   \, r_{j_0}^{\frac {5}{6}+\eps}
\eqn
uniformly in $x \in \Gamma_\chi \bsl G$. Thus, we have shown our first main result.
\begin{thm}
Let $G=\SL(2,\R)$ and $\Gamma$ be a congruence arithmetic lattice in $G$. {Let further  $\chi \in \widehat \Gamma$ be a Nebentypus character. Then, for any Hecke--Maass form $\phi_j \in \L^2_{\chi}(\Gamma\bsl G)$ with $\Delta \phi_j =\lambda_j \phi_j$  and $\norm{\phi_j}_{\L^2}=1$ one has
\bqn 
\norm{\phi_j}_\infty \ll_\eps     \lambda_j^{\frac{5}{24}+\eps}
\eqn
for any $\eps>0$. }
\label{thm:24.11.2016}
\end{thm}
\qed

\medskip

For trivial $\sigma_l$ and $\chi$, Theorem \ref{thm:24.11.2016} is due to Iwaniec-Sarnak{ \cite[Theorem 0.1]{iwaniec-sarnak95}}. In fact, their method 
can also be used for non-trivial $\chi$, as was done in \cite[Section 10]{BH} for non-compact arithmetic surfaces. Thus, we recover

\begin{cor}[\bf Iwaniec-Sarnak]
For any Hecke--Maass form $\phi_j \in \L^2_\chi(\Gamma \bsl \mathbb{H})$  with $\norm{\phi_j}_{\L^2}=1$ and Beltrami--Laplace eigenvalue $\lambda_j$ one has
\bqn 
\norm{\phi_j}_\infty \ll_\eps \lambda_j^{\frac5{24}+\eps}
\eqn
for any $\eps>0$. 
\label{thm:iwaniec-sarnak95}
\end{cor}
\begin{proof}
If $\sigma_l=\id$ is trivial, $\L^2_{\sigma_l,\chi}(\Gamma \bsl G) \simeq \L^2_\chi(\Gamma \bsl \mathbb{H})$, and the assertion follows form the previous theorem. Note that since all $K$-orbits in $G$ have the same volume, each eigenfunction of the Beltrami--Laplace operator on $\bH\simeq G/K$  lifts to a unique $K$-invariant eigenfunction of the Beltrami--Laplace operator on $G$.
\end{proof}

 \medskip

\subsection{Automorphic forms on $\SL(2,\R)$ and representation-theoretic interpretation} 
\label{sec:autrepr}
In what follows, we would like to discuss our results within the  theory of automorphic forms and their  representa\-tion-theoretic meaning.  For this, let us first recall the concept of an automorphic form on $G=\SL(2,\R)$ for a discrete co-compact subgroup $\Gamma$, compare \cite[Section 5]{borel97}.

\begin{definition}
\label{def:autom}
A smooth function $f:G \rightarrow \C$ is called an \emph{automorphic form on $G$ for $\Gamma$} iff:
\begin{enumerate}
\item[(A1)] $f(\gamma g)=f(g)$ for all $\gamma \in \Gamma$ and $g \in G$,
\item[(A2)] $f$ is $K$-finite on the right, where  $K=\SO(2)$,
\item[(A3)] $f$ is $\mathcal{Z}$-finite, where $\mathcal{Z}$ denotes the center of the universal envelopping algebra $\mathfrak{U}(\g_\C)$ of the complexification of $\g$.
\end{enumerate}
\end{definition}

Note that  (A2) means that $f$ is a finite sum of functions $f_l$ belonging to  a specific $K$-type $\sigma_l$, while (A3) is equivalent to the existence of a polynomial $p(\Ccal)$ in the Casimir operator $\Ccal=dR(\Omega)$ that annihilates $f$, the notation being as in Section \ref{sec:hecke}. Now, $\g=\{X\in M(2,\R) \mid \tr(X)=0 \}$, while  a Cartan involution is given by $-\, ^tX$. With respect to the basis
\bqn 
X_1':=\begin{pmatrix} 1& 0 \\ 0 & -1\end{pmatrix}, \qquad X_2':=\begin{pmatrix} 0& 1 \\ 1 & 0 \end{pmatrix}, \qquad Y_1':=\begin{pmatrix} 0& 1 \\ -1 & 0 \end{pmatrix}
\eqn
of $\g=\p\oplus \k$, the modified Killing form $\eklm{\cdot, \cdot}_\theta$ is represented by the matrix
\bqn 
 \frac 18 \begin{pmatrix} 1 & 0 & 0 \\ 0 & 1 & 0 \\ 0 & 0 & 1\end{pmatrix}.
\eqn
Consequently, a corresponding orthonormal basis of $\g$ is given by $X_1:=X_1'/2 \sqrt{2}$, $X_2:=X_1'/2 \sqrt{2}$, $Y_1:=Y_1'/2 \sqrt{2}$, so that the Casimir element reads
\bqn 
\Omega=X_1^2+X_2^2-Y_1^2\equiv \frac 1 8 \id,
\eqn
compare \eqref{eq:casimir}. Note that our normalization of $\Ccal$ differs from the one in \cite[p. 20]{borel97},  where $\Omega\equiv \frac 12 \id$.
  Writing $p(\Ccal)=\prod_i(\Ccal-\mu_i)$ and $\mu_{\sigma_l}=l^2/8$ for the eigenvalue of $dR(\Omega_K)$ on the $\sigma_l$-isotypic component one sees that
\begin{align*}
p(\Ccal)f_l=\prod_i(2dR(\Omega_K)-\Delta-\mu_i)f_l =\prod_i(2\mu_{\sigma_l}-\Delta-\mu_i)f_l =:q_l(\Delta) f_l,
\end{align*}
where we took into account \eqref{eq:casimir-laplace}. Thus, $p(\Ccal) f=0$ iff $q_l(\Delta)f_l=0$ for all $l$, by orthogonality.   Since $q_l(\Delta)$ is an elliptic differential operator of the same order than $p(\Ccal)$, and any subspace defined by a $K$-type and a Casimir eigenvalue is finite dimensional by  Harish-Chandra's theorem \cite[Theorem 1.7]{borel-jacquet}, we see that $f$ is essentially given by a finite sum of  Hecke--Maass forms in the sense of this paper.

To interprete our results in terms of the representation theory of $G$, let us first notice that, since $-I_2$ belongs to $\Gamma$, one has 
\bqn 
\L^2_{\sigma_l,\chi}(\Gamma \bsl G)=\mklm{0} \quad \text{ if } \quad \sigma_l(-I_2)\neq \chi(-I_2).
\eqn
Hence, in case that  $\L^2_{\sigma_l,\chi}(\Gamma \bsl G)\not=\mklm{0}$,  $l$ must be even if $\chi(-I_2)=1$,  and odd otherwise. Now, according to the decomposition \eqref{eq:Gdecomp} the following irreducible unitary representations of $G$ can appear in $\L^2_{\sigma_l,\chi}(\Gamma \bsl G)$, see  \cite[Section 15]{borel97}:

\begin{enumerate}
\item If $l=0$, 

\vspace{1mm}
(a) the trivial representation,

\vspace{1mm}
(b) the unitary principal series $H(0,s)$ with $s\in i\R_{\geq 0}$,

\vspace{1mm}
(c) the complementary series $I(0,s)$ with $s\in (0,1)$

\noindent
can appear.
\item If $l$ is even and $l\neq 0$,

\vspace{1mm}
(a) the discrete series $D_s$ with $s \in \Z-\mklm{0}$, $|s|<|l|$, $\sgn(s)=\sgn(l)$, $s$ odd,

\vspace{1mm}
(b) the unitary principal series $H(0,s)$ with $s\in i\R_{\geq 0}$,

\vspace{1mm}
(c) the complementary series $I(0,s)$ with $s\in (0,1)$

\noindent
can appear. 
\item If $l$ is odd, 

\vspace{1mm}
(a) the discrete series $D_s$ with $s \in \Z-\mklm{0}$, $|s|<|l|$, $\sgn(s)=\sgn(l)$, $s$ even,

\vspace{1mm}
(b) the unitary principal series $H(1,s)$ with $s\in i\R_{> 0}$,

\vspace{1mm}
(c) the limits of discrete series $D_{+,0}$ (resp. $D_{-,0}$) with $l>0$ (resp. $l<0$)

\noindent
can appear.
\end{enumerate}

Note that in each of the above unitary representations the $\sigma_l$-isotypic component is $1$-dimensional. 
By the above list, one sees that representations occuring in $\L^2_{\sigma_l,\chi}(\Gamma \bsl G)$ are in general different from the spherical case $\L^2(\Gamma\bsl \bH)$. Hence, Theorem \ref{thm:24.11.2016} implies subconvex bounds for new classes of automorphic representations, in particular for the discrete series $D_s$ and their limits $D_{\pm,0}$, as well as the principal series $H(1,s)$.

To rephrase our results in terms of the spectrum of the Casimir operator $\Ccal$, note that by \eqref{eq:casimir-laplace} we have with $\mu_{\sigma_l}=l^2/8$
\bq
\label{eq:29.8.2017}
\Delta \phi_j =  \lambda_j \phi_j=\big (-\mu_j+{l^2}/4\big ) \phi_j, \qquad \phi_j \in \L^2_{{\sigma_l},\chi}(\Gamma\bsl G), 
\eq
 $\mu_j$ being the Casimir eigenvalue of $\phi_j$. Now, in the cases relevant to us the Casimir eigenvalues can take the following values  \cite[p.\ 158 and p.\ 163]{borel97}:\footnote{Note that our normalization of the Casimir operator differs from the one in \cite{borel97} by a factor $1/4$.}

\begin{enumerate}

\item On $(D_s)^\infty$,  $\mathcal{C}\equiv(s^2-1)\, \id /8$, where $s \in \Z-\mklm{0}$. For $|s|<|l|$ and $\sgn(s)=\sgn(l)$ one has $0<s^2-1<l^2-1$.

\item On $H(0,s)^\infty$ or $H(1,s)^\infty$,  $\mathcal{C}\equiv (s^2-1)\, \id /8$. For $s\in i\R_{\geq 0}$ one has $s^2-1\leq -1$.

\item On $I(0,s)^\infty$,  $\mathcal{C}\equiv (s^2-1)\, \id /8$. For $s\in (0,1)$ one has $-1<s^2-1<0$.

\item On $(D_{+,0})^\infty$ or $(D_{-,0})^\infty$, ${\displaystyle \mathcal{C}\equiv- \id }/8$.

\end{enumerate}
 
Here $H^\infty$ denotes the subspace of differentiable vectors in a Hilbert representation $H$. Consequently, the subconvex bound in Theorem \ref{thm:24.11.2016} can be restated as follows.

\begin{thm}
\label{thm:28.8.2017}
Let $G=\SL(2,\R)$, $K=\SO(2)$, and $\Gamma$ be a congruence arithmetic lattice in $G$. Let further ${\sigma_l} \in \widehat K$ and  $\chi \in \widehat \Gamma$ be a Nebentypus character. {Then,  for any Hecke eigenform $\phi \in \L^2_{{\sigma_l},\chi}(\Gamma\bsl G)$ satisfying $\norm{\phi}_{\L^2}=1$ and  $\Ccal \phi = \frac{s^2-1}8\phi$ one has 
 \bqn 
\norm{\phi}_\infty \ll_\eps     ( 1-s^2+2l^2 )^{\frac{5}{24}+\eps} 
\eqn}
for any $\eps>0$.
\end{thm}
\qed


Classically, an \emph{automorphic form of weight $l\in \N$ and Nebentypus character $\chi$}  was first introduced as a holomorphic function  $f:\bH \rightarrow \C$ satisfying
\bqn 
f(\gamma \cdot z)= \chi(\gamma)\, j(\gamma,z)^l \, f(z), \qquad \gamma \in \Gamma, \, z \in \bH,
\eqn
where  $j(g,z)$ is as in \eqref{eq:j(g,z)}. Its lift $\tilde f(g):=f(g\cdot i) j(g,i)^{-l}$ constitutes an automorphic form on $G$ in the sense of Definition \ref{def:autom}; it  is  of $K$-type $\sigma_l$ and satisfies $\Ccal \tilde f=\frac 14 (l^2/2-l) \tilde f$, see  \cite[Sections 5.14 and 5.15]{borel97}.
In particular, if $l>1$,  $\tilde f$ belongs to the discrete series representation $D_{l-1}$ in $\L^2_{{\sigma_l},\chi}(\Gamma\bsl G)$. 
If, in addition,  $\tilde f$ is a  Hecke eigenform with $\|\tilde f\|_2=1$, one deduces from Theorem \ref{thm:28.8.2017} 
{
\bq\label{eq:3.9.2018}
\|{ f}\|_\infty \ll_\eps l^{\frac{5}{12}+\eps},
\eq
since  $\norm{f}_p\equiv \|\tilde f\|_p$ for all $p$, compare \cite[p. 219]{Miyake}, yielding subconvex bounds for classical automorphic forms on $\bH$ in the weight aspect.  This is consistent with  Godement's formula \cite{Godement}, by which one has the convex bound $\norm{f}_\infty \ll l^{\frac12}$, see \cite{DS, CL}.}
Furthermore, a corresponding subconvex bound was proven in \cite{DS}, the exponent there being $\frac{1}{2}-\frac{1}{33}=\frac{31}{66}$. Thus, our results {do}  imply new results about holomorphic modular forms on $\bH$.  Note that in  the case $\Gamma=\SL(2,\Z)$ one can even   show  \cite{Xia} that $l^{\frac{1}{4}-\eps}\ll_\eps \norm{f}_\infty \ll_\eps l^{\frac{1}{4}+\eps}$ by using  the Fourier expansion of $f$ and Deligne's bound \cite{Deligne}, though this method is not available for  cocompact arithmetic subgroups.
In the non-cocompact case, hybrid bounds in the eigenvalue and the level aspect were considered in \cite{BH}.

\section{Subconvex bounds on $\SO(3)$}
\label{sec:SO(3)}

In this section, we shall derive equivariant and non-equivariant subconvex bounds on $\SO(3)$ in the setting of \cite{LPS1,LPS2}. They are proven in an analogous way than the ones proven in Section \ref{sec:IS} using results of  \cite{VanderKam}. To begin, consider the  quaternion algebra
\[
H(R):=\mklm{a_0+a_1\bi+a_2\bj+a_3\bk\mid a_0,a_1,a_2,a_3\in R  }
\]
over a given commutative ring $R$, where $\bi^2=\bj^2=-1 , \,\bi\bj=-\bj\bi=\bk$, and recall that for an element $a:=a_0+a_1\bi+a_2\bj+a_3\bk\in H(R)$ its conjugate   is given by $\overline{a}:=a_0-a_1\bi-a_2\bj-a_3\bk\in H(R)$ while its norm reads $N(a):=a\overline{a}=a_0^2+a_1^2+a_2^2+a_3^2$.
Note that  $H(\R)$ corresponds to the field of Hamilton's quaternions and $H(\Z)$ to the ring of Lipschitz integers \cite{CS}. Write $H(R)^1:=\{a\in H(R)\mid N(a)=1\}$ and put $G:=H(\R)^1$. As a  group $G$, can be identified with $\SU(2)$ via the mapping
\[
G\ni a_0+a_1\bi+a_2\bj+a_3\bk \quad \longmapsto  \quad \begin{pmatrix}a_0+a_1i & a_2+a_3i \\ -a_2+a_3i& a_0-a_1i \end{pmatrix}\in \SU(2).
\]
$G$  is compact,  while $H(\Z)^1=\{\pm 1,\;\; \pm \bi,\;\; \pm \bj,\;\; \pm \bk \}$ is finite, so that by  choosing the lattice
$\Gamma :=  \{\pm 1\}$ in $G$ 
 we have $$\Gamma\bsl G\cong \SO(3)$$ via the adjoint action of $G$ on its Lie algebra. Next, we introduce  Hecke operators on $\SO(3)$ following \cite{LPS1,LPS2}. Thus, for each $\alpha\in H(\R)\setminus \{0\}$ and $x\in G$ set
\[
\alpha\cdot x:= N(\alpha)^{-1/2}\alpha x \in G.
\]
As in  Section \ref{sec:hecke}, one can associate to each double coset $\Gamma \alpha \Gamma$, $\alpha\in H(\Q)$, a Hecke operator $T_{\Gamma\alpha\Gamma}$. 
We then have $T_{\Gamma\alpha_1\Gamma}\circ T_{\Gamma\alpha_2\Gamma}=T_{\Gamma\alpha_2\alpha_1\Gamma}$. Further, setting \footnote{$a_0+a_1\bi+a_2\bj+a_3\bk\equiv 1 \; (2)$ means that $a_0$ is odd and $a_1,a_2,a_3$ are even. Note that $\mathcal{R}(n)$ is empty unless $n\equiv 1\mod 4$.}
\[
\mathcal{R}(n):=\{ a\in H(\Z) \mid N(a)=n ,\quad a\equiv 1 \mod 2 \},
\]
one can define the Hecke operator
\[
(\T_nf)(x):=\frac{1}{2}\sum_{\alpha\in \mathcal{R}(n)}f(\alpha\cdot x), \qquad f\in \L^2(\Gamma\bsl G).
\]
For natural numbers $r$, $s\equiv 1\mod 4$, one has $\T_r \T_s=\sum_{d|(r,s)}d \, \T_{rs/d^2}$, see \cite[Remark 1]{LPS2} and \cite[p. 331]{VanderKam}. Since Hecke operators commute with the right regular representation of $G$ on $\L^2(\Gamma\bsl G)$, we may replace $\L^2(\Gamma\bsl G)$ by $\L^2(\Gamma\bsl G/K)$ for any subgroup $K$ of $G$ in the above argument on Hecke operators. Choose $K=\SO(2)\simeq \C^1$, and  denote the corresponding characters by $\sigma_l:e^{i\theta} \mapsto e^{il\theta}$, $l \in \Z$.
Let $\Delta$ denote the Beltrami--Laplace operator on $G$.
Since $\Delta$ and $\T_n$ commute, there exists  an orthonormal basis $\mklm{\phi_j}_{j\geq 0}$ of $\L^2(\Gamma\bsl G)$ consisting of simultaneous eigenfunctions compatible with the decompositions \eqref{eq:directsum} and  \eqref{eq:PW}, where $X=G$ and $M=\Gamma \bsl G$, respectively. Further, note that  the action of $K$ on $\Gamma \bsl G$ is  isometric and non-singular. We then can prove the following equivariant subconvex bounds. 

\begin{thm}\label{thm:SO(3)equiv}
Let $G=\SU(2)$, $K=\SO(2)$, $\Gamma=\{\pm 1\}$, and $\sigma_l\in\widehat{K}\simeq \Z$. Then, 
for any Hecke--Maass form  {$\phi_j\in \L_{\sigma_l}^2(\Gamma\bsl G)$
with  Beltrami--Laplace eigenvalue $\lambda_j$ and $\norm{\phi_j}_{\L^2}=1$ one has
\[
\norm{\phi_j}_\infty \ll_\eps \, \lambda_j^{\frac{5}{24}+\eps},
\]}
$\eps>0$ being arbitrary.
\end{thm}
\begin{rem}
This theorem  is a generalization of \cite[Theorem 1.1]{VanderKam}, where the case  $\L^2_{\sigma_0}(\Gamma\bsl G)\cong \L^2(\Gamma \bsl G/K) \cong \L^2(S^2)$ is treated, $S^2$ being the $2$-sphere. 
In the papers \cite{BMa,BMb}, hybrid $\L^\infty$-norms for general arithmetic quotients of $2$-spheres  in the eigenvalue and  level aspect are studied.
However, the exponents for the spectral parameter are a little greater than $5/24$ there, namely  $\frac{1}{24}-\frac{1}{27}=\frac{1}{216}$.
\end{rem}
\begin{proof}
By \cite[Lemma 2.1]{VanderKam} one has 
\bq \label{VanderKam}
\sharp\{\alpha\in\mathcal{R}(n)\mid \text{dist}(x,\alpha\cdot x)<\delta\}\ll_\eps \begin{cases} \delta^{\frac{1}{2}}n^{1+\eps}+n^\eps & \text{if $\delta<1/n$,} \\ n^{\frac{1}{2}+\eps}+\delta^{\frac{2}{3}}n^{1+\eps} & \text{otherwise.} \end{cases}
\eq
Further,  Proposition \ref{thm:15.09.2016} also holds in the present case,  
since  Proposition \ref{prop:30.01.2017} is true  for arbitrary compact manifolds and symmetry groups. By repeating the arguments given in Section \ref{sec:Equivariant subconvex bounds} we therefore get {for any  $|l| \ll \mu /\log \mu$,  $\mu>n$, and $n\equiv 1\mod 4$ the uniform bound
\[
K_{\T_n \circ \widetilde s_\mu \circ \Pi_{\sigma_l} }(x,x) \ll_\eps \,  (\mu+ n \mu^{1/2}\log\mu) \, n^\eps
\]
up to neglegible terms.} Now, by  the Dirichlet prime number theorem on arithmetic progressions it is well-known that
\bqn
\#\mklm{ p<x \mid \text{$p$ is a prime}, \, p\equiv 1\mod 4 } \; \sim \;  \frac{1}{2}\frac{x}{\log x}. 
\eqn
Furthermore,  the Ramanujan conjecture proved by Deligne \cite{Deligne} together with  the Jacquet-Langlands correspondence for $\GL(2)$ \cite{KR} implies that the Hecke eigenvalues $\lambda_j(p)$ of $\T_p$ are bounded from above by $2p^{\frac{1}{2}+\eps}$  for prime levels. Hence,  the argument of Iwaniec-Sarnak  already used  in Section \ref{sec:Equivariant subconvex bounds}, but now applied to  {$\L^2(\SO(3))$,    yields 
\[
\norm{\phi_j}_\infty^2 \ll_\eps \, \mu^\eps N^\eps (\mu N^{-1/2}+ \mu^{1/2}N).
\]}
The theorem now follows by taking $N= \mu^{1/3}$.
\end{proof}

Note that the right regular representation of $\Gamma \bsl G \cong \SO(3)$ on $\L^2(\Gamma\bsl G)$ decomposes according to 
\[
\L^2(\Gamma\bsl G)\cong \bigoplus_{k\in \N} \mathcal{M}_k, \qquad \mathcal{M}_k\cong \pi_k^{\oplus 2k+1},
\]
 where $\pi_k$ denotes the irreducible representation of $\SO(3)$ of dimension $2k+1$.
In particular, the Beltrami--Laplace eigenvalue corresponding to $\pi_k$ is $k(k+1)$, and the restriction of $\pi_k$ to $K$ is isomorphic to  $\widehat\bigoplus_{l=-k}^{k} \sigma_l$.
{Hence, if we choose an orthonormal sequence $\{\psi_j\}_{j\geq 0}$ in $\L^2(\Gamma\bsl G)$ consisting of Hecke--Maass forms with $\psi_j\in (\sigma_{l_j})^{\oplus 2k_j+1} \subset \mathcal{M}_{k_j}$, where $| l_j |\leq k_j$,  Theorem \ref{thm:SO(3)equiv} yields
\[
\norm{\psi_j}_\infty\ll_\eps  k_j^{\frac{5}{12}+\eps}.
\]
}

\section{Subconvex bounds on $\Gamma\bsl G$ for semisimple groups and arithmetic congruence subgroups}
\label{sec:6}

\subsection{General framework}

 In what follows, we shall develop  a general framework to prove subconvex bounds of Hecke--Maass forms on semisimple groups.
For this, let us return to the general setting of Sections \ref{sec:hecke}, \ref{sec:FIO} and \ref{sec:4}.
Write  $\mathfrak{L}^2:=\L^2_{\chi}(\Gamma\bsl G)$ or $\L^2_{\sigma,\chi}(\Gamma\bsl G)$, and consider on this space the family of  Hecke operators $T^\chi_{\Gamma\beta\Gamma}$  introduced in \eqref{eq:heckechar}, together with the corresponding  $\C$-module 
\[
H_\Xi^\chi:=\langle \T^\chi_{\Gamma\beta\Gamma} \mid \beta\in\Xi\rangle
\]
generated by them.  In what follows, we assume that there exists a submodule $\mathcal{H}$ of $H_\Xi^\chi$  such that there is  an orthonormal basis $\{\phi_j\}_{j\in\N}$ of $\L^2(\Gamma_\chi\bsl G)$ compatible with the decomposition \eqref{eq:directsum},  and  in case that $P$ commutes with the right regular $K$-representation, also with the decomposition \eqref{eq:PW}, consisting of  simultaneous eigenfunctions of  $P$ and all $\T \in \mathcal{H}$ with $P \phi_j=\lambda_j \phi_j$. As before, such simultaneous eigenfunctions will be  called \emph{Hecke--Maass forms of rank $1$}. We also suppose that $\T^*$ belongs to $\mathcal{H}$ for each $\T\in\mathcal{H}$ and that the cospheres $S^\ast_x(\Gamma_\chi \bsl G):=\mklm{(x,\xi) \in T^\ast (\Gamma_\chi \bsl G)\mid  p(x,\xi)=1}$ are strictly convex for all $x \in \Gamma_\chi\bsl G$.
 Further,  consider the lattice point counting functions $\mathfrak{M}(x,\beta, \delta):=\M(x,\beta, \delta)$ or $M(x,\beta, \delta)$ {corresponding to $\mathfrak{L}^2$, respectively; that is, 
\bqn 
\M(\delta):=\M(x,\beta,\delta):=\#\mklm{\alpha \in \Gamma_\chi \backslash \Gamma \beta \Gamma:  \dist(x,\alpha \cdot x)<\delta},
\eqn
where $\dist(\alpha \cdot x,x) \equiv  \dist (\Gamma \alpha x, \Gamma x)$, and $M(x,\beta, \delta)$ is as in  \eqref{eq:12.02.2017}.} We then have the following 

\begin{lem}
\label{lem:general}
Fix a character $\chi$ in $\widehat\Gamma$ such that $[\Gamma:\Gamma_\chi]<\infty$. Let $\phi_{j_0}$ be a Hecke--Maass form in $\mathfrak{L}^2$ with corresponding spectral eigenvalue $\lambda_{j_0}$.
Let $\mathcal{P}'$ be an infinite set and $\mathfrak{N}':\mathcal{P}'\to \N$ a mapping such that
\bq \label{eqnote1}
\#\mklm{v\in \mathcal{P}' \mid  N/2< \mathfrak{N}'(v)< N} \gg N/\log N.
\eq
Assume that for each element $v\in\mathcal{P}'$   there exists a Hecke operator $\T_v'\in\mathcal{H}$ satisfying $\T'_v\phi_{j_0}=\phi_{j_0}$, and that  for any $N \in \N$ and any $x\in \Gamma_\chi\bsl G$  we have a suitable finite subset $Q_{N,x}'\subset \mathcal{P}'$ such that $\# Q_{N,x}'\ll \log N$. Write
\[
 \T_{N,x}':=\sum_{v\in\mathcal{Q}_{N,x}'}\T_v'  \quad \text{where} \quad \mathcal{Q}_{N,x}':=\{v\in\mathcal{P}' \mid N/2<  \mathfrak{N}'(v)< N, \;\; v\not\in Q_{N,x}' \}.
\]
As a linear operator  on $\mathfrak{L}^2$, $\T'_{N,x}\circ (\T_{N,x}')^*$ can be represented  as 
\[ 
\T'_{N,x}\circ (\T_{N,x}')^*=\sum_{u=1}^l a_u \T_{\Gamma\alpha_u\Gamma}^\chi 
\]
for certain $l \in \N$, $a_u\in\C$, and $\alpha_u \in \Xi$ depending on $x$.
Further, suppose that there exist numbers $0 < \kappa\ll 1$ and $A_1$, $A_2>2$ such that for each $N\gg 1$ and each $x\in \Gamma_\chi\bsl G$ one has
\bq
\label{eq:19.03.2017}
\sum_{u=1}^l |a_u||\mathfrak{M}(x,\alpha_u,N^{-A_2})|\ll N^{2-2\kappa},\quad \sum_{u=1}^l |a_u||\Gamma_\chi\bsl \Gamma\alpha_u\Gamma|\ll N^{A_1}.
\eq
Then, if  $\mathfrak{L}^2=\L^2_\chi(\Gamma\bsl G)$, there exists a constant $\delta>0$ such that
\[
\|\phi_{j_0}\|_\infty \ll   \lambda_{j_0}^{\frac{\dim G-1}{2m}-\delta},
\]
while if $\mathfrak{L}^2=\L^2_{\sigma,\chi}(\Gamma\bsl G)$, there exists a constant $\delta>0$, which does not depend on $\sigma$, such that
\[
\|\phi_{j_0}\|_\infty \ll \sqrt{ d_\sigma \max_{u \leq \lfloor \frac{\dim K}2+1\rfloor} \norm{D^u \sigma}_\infty }\, \,  \lambda_{j_0}^{\frac{\dim G/K-1}{2m}-\delta}.
\]
{Finally,  if $K=T$ is a torus and  $\mathfrak{L}^2=\L^2_\chi(\Gamma\bsl G)$, there exists a constant $\delta>0$ such that
\[
\|\phi_{j_0}\|_\infty \ll   \lambda_{j_0}^{\frac{\dim G/K-1}{2m}-\delta}.
\]
}

\end{lem}

\begin{proof}
Let us consider first the case  $\mathfrak{L}^2=\L^2_{\sigma,\chi}(\Gamma\bsl G)$. Set $\mu:=\sqrt[m]{\lambda_{j_0}}$
and denote by $\lambda'_{j,N}$  the eigenvalue of $\T_{N,x}'$ for $\phi_j$, so that 
\bqn 
|\T_{N,x}' \phi_j(x)|^2= \lambda'_{j,N}\overline{\lambda'_{j,N}}\phi_j(x) \overline{\phi_j(x)}=\T_{N,x}'\circ(\T_{N,x}')^*\phi_j(x)\overline{\phi_j(x)}.
\eqn
Taking $\tilde N=0$ in Proposition \ref{prop:30.01.2017}  we deduce that
\[
\sum_{u=1}^l a_u  \sum_{\stackrel{\alpha\in\Gamma_\chi\bsl \Gamma\alpha_u\Gamma,}{\dist (xK,\alpha \cdot xK)^{\dim G/K-1} \geq N^{-A_2}}} \overline{\chi(\alpha)}  K_{\tilde s_\mu \circ \Pi_\sigma}(\alpha \cdot x,x) \ll d_\sigma\max_{u \leq \lfloor \frac{\dim K}2+1\rfloor} \norm{D^u \sigma}_\infty \mu^{\frac{\dim G/K-1}{2}}  N^A 
\]
up to terms of order $O(N^{A} \, \mu^{-\infty})$, 
where we set $A:=\frac{1}{2}A_2+A_1$. With the same arguments than at the end of Section \ref{sec:FIO} and in  the proof of Proposition \ref{thm:15.09.2016} one now deduces with $\mu_j:= \sqrt[m]{\lambda_j}$ for any $x \in \Gamma_\chi \bsl G$
\begin{align*}
(\# \mathcal{Q}_{N,x}')^2& |\phi_{j_0}(x)|^2  = |\T_{N,x}' \phi_{j_0}(x)|^2 \leq \sum_{\stackrel{\mu \leq \mu_j \leq \mu+1,}{\phi_j \in \L^2_\sigma(\Gamma_\chi \bsl G)}} |\T_{N,x}' \phi_j(x)|^2 \\
& \ll \sum_{\stackrel{j \geq 0,}{\phi_j \in \L^2_\sigma(\Gamma_\chi \bsl G)}} \rho(\mu-\mu_j)  |\T_{N,x}' \phi_j(x)|^2   = K_{\T_{N,x}'\circ (\T_{N,x}')^* \circ \widetilde s_\mu \circ \Pi_\sigma }(x,x)  \\
& \ll  d_\sigma  \max_{u \leq \lfloor \frac{\dim K}2+1\rfloor} \norm{D^u \sigma}_\infty \mu^{\dim G/K-1}  N^{2-2\kappa}  + d_\sigma\max_{u \leq \lfloor \frac{\dim K}2+1\rfloor} \norm{D^u \sigma}_\infty \mu^{\frac{\dim G/K-1}{2}}  N^A \\
& = d_\sigma \max_{u \leq \lfloor \frac{\dim K}2+1\rfloor} \norm{D^u \sigma}_\infty \Big( \mu^{\dim G/K-1}  N^{2-2\kappa} + \mu^{\frac{\dim G/K-1}{2}}  N^A  \Big)  
\end{align*}
up to terms of order $O(N^{A} \, \mu^{-\infty})$. Hence, the assertion follows from \eqref{eqnote1} by taking $N\sim \mu^B$,  $B:=\frac{\dim G/K-1}{2(A-2+2\kappa)}$. The case $\mathfrak{L}^2=\L^2_{\chi}(\Gamma\bsl G)$ is seen in a similar way taking into account Proposition \ref{prop:31.05.2017} {and the toric case in Proposition \ref{prop:30.01.2017}. }
\end{proof}

\begin{rem}
The assumptions of the previous lemma are primarily motivated by the work of Marshall \cite{Marshall2017} in the case that $\chi$ is trivial. One can easily verify that they are fulfilled in the setup of Section \ref{sec:IS} when $\mathcal{P}'$ is the totality of primes and $\mathfrak{N}'$ is the inclusion mapping $\mathcal{P}' \subset \N$.
Indeed, for each prime $p$ with $(p,q)=1$, there exists an element $\beta_p$ in $\cR(p^2)$ such that $\tilde \T^\chi_{p^2}:=\T^\chi_{\Gamma\beta_p\Gamma}=\T^\chi_{p^2}-\T^\chi_{\Gamma(pI_2)\Gamma}$, see \cite[p.\  217]{Miyake}. Now, let $\tilde\lambda_j(p^2)$ be the eigenvalue of $\tilde \T^\chi_{p^2}$ belonging to the eigenfunction $\phi_j \in \L^2_{\sigma,\chi}(\Gamma  \bsl G)$. For each $\phi_j$ in $\L^2_{\sigma,\chi}(\Gamma  \bsl G)$ and $p$ as above,  $|\lambda_j(p)|\leq p^{1/2}/2$ implies that  $|\tilde\lambda_j(p^2)|\geq p/2$ by $|\tilde\lambda_j(p^2)|=|\lambda_j(p)^2-(p+1)\chi(p)|$. Therefore, if we choose
\[
\T_p':=\begin{cases} \frac{1}{\lambda_{j_0}(p)} \T^\chi_p & \text{if $|p^{-1/2}\lambda_{j_0}(p)|>1/2$}, \\ \frac{1}{\tilde\lambda_{j_0}(p^2)} \tilde \T^\chi_{p^2} & \text{otherwise},\end{cases} 
\]
 the mentioned assumptions must hold by Lemma \ref{lem:05.09.2016}.
This choice is essentially the same as in the work of Blomer-Maga \cite{BMa,BMb}  on  $\SL(n,\Z)\subset \PGL(n,\R)$ in the case $n=2$. 
\end{rem}

{
\begin{rem}\label{rem:rank}
When using Lemma \ref{lem:general},  we shall take for $\mathcal{P}'$ a subset of the totality of primes. Note that it is unnecessary to suppose that $\phi_j$ is an eigenfunction of the operators $\T_p$ for all primes $p$ to prove the subconvex bound in Section \ref{sec:IS}, as we did already see in Section \ref{sec:SO(3)}.
In fact, one can make the conditions on $\mathcal{P}'$ and $\mathcal{H}$ weaker.
Namely, we can replace $\mathcal{P}'$ by a smaller subset satisfying \eqref{eqnote1} and replace $\mathcal{H}$ by the submodule $\langle  \T_v' \mid v\in \mathcal{P}'  \rangle$.
This means that our concept of a Hecke-Maass form is much weaker than the usual one in Section \ref{sec:IS}.
Such forms are not eigenfunctions of the center of the universal enveloping algebra in general,  and can be obtained in abundance by functorial lifts of Hecke characters.
\end{rem}
}

\subsection{Equivariant subconvex bounds}
\label{sec:6.2}

In what follows, we shall derive  equivariant subconvex bounds on arithmetic quotients  for a large class of semisimple algebraic groups, extending the work of Marshall \cite{Marshall2017} to non-spherical situations. Thus, let $\uG$ be a connected semisimple algebraic group over a number field $F$.
We write $\uG(k)$ for the set of $k$-rational points in $\uG$ for a field $k\supset F$ and $F_v$ for the completion of $F$ by a place $v$ of $F$.
Following  \cite{Marshall2017}, we assume that there exists a real place $v_0$ of $F$ such that

\medskip

\begin{itemize}
\item[(WS)] The group $\uG(F_{v_0})$ is quasi-split, and not isogeneous to a product of odd special unitary groups.
\end{itemize}

\medskip

We set $H:=\mathrm{Res}_{F/\Q}\uG$, where $\mathrm{Res}_{F/\Q}$ means the restriction of scalars from $F$ to $\Q$.
Then $H$ is a connected semisimple algebraic group over $\Q$ and $G:=H(\R)$  a real semisimple Lie group with finite center \cite[Chapter 3]{PR}.
Let $K$ be a maximal compact subgroup in $G$,  $K_0$  an open compact subgroup of $H(\bA_\mathrm{fin})$, and put  $\Gamma:=H(\Q)\cap(H(\R)K_0)$. In the following theorem, we will also impose the condition $H(\bA)=H(\Q)(H(\R)K_0)$. This condition holds for any $K_0$ if $H$ has the strong approximation property with respect to $\infty$, that is, if $H$ is simply connected {as an  algebraic group,  and does not have any $\R$-simple component $H'$ such that  $H'(\R)$ is  compact, compare \cite[Theorem 7.12]{PR}.
As the second main result of this paper we obtain 
\begin{thm}
\label{thm:general1} 
Suppose that $H(\bA)=H(\Q)(H(\R)K_0)$ and $H(\Q)\bsl H(\bA)$ is compact, so that $\Gamma \bsl G$ is compact as well. Further, assume that $H$ satisfies the condition {\rm (WS)}. Now, let  $P_0$ be an elliptic left-invariant differential operator on $G$ of degree $m$ which gives rise to a  positive and symmetric operator on $\Gamma \bsl G$  that commutes with the right regular $K$-representation and has strictly convex cospheres $S^\ast_x(\Gamma \bsl G)$.
Then, there exist a submodule $\mathcal{H}$ of $H^{\chi=1}_{\Xi=H(\Q)}$ and a constant $\delta>0$  such that
\begin{enumerate}
\item there is an orthonormal basis $\{\phi_j\}_{j\in\N}$ of $\L^2(\Gamma\bsl G)$ which consists of simultaneous eigenfunctions for the unique self-adjoint extension $P$ of $P_0$ and all $\T \in \mathcal{H}$;
\item for each $\phi_j\in \L^2_\sigma(\Gamma\bsl G)$ with spectral eigenvalue $\lambda_j$ one has 
\[
\|\phi_j\|_\infty \ll \sqrt {d_\sigma \sup_{u \leq \lfloor \frac{\dim K}2+1\rfloor} \norm{D^u \sigma}_\infty} \, \, \lambda_j^{\frac{\dim G/K-1}{2m}-\delta}.
\] 
{If $K=T$ is a torus, one has the stronger estimate
\[
\|\phi_j\|_\infty \ll  \, \lambda_j^{\frac{\dim G/K-1}{2m}-\delta}.
\] }
\end{enumerate}
\end{thm}

\begin{rem}
\label{rem:22.5.2017}
The equivariant   subconvex bound of the previous theorem can be rephrased using the Cartan-Weyl classification  of unitary irreducible representations of compact groups. In fact, assume that  $K$ is a compact connected semisimple Lie group, $\k$ its Lie algebra,  and $T\subset K$ a maximal torus with Lie algebra $\t$. 
Denote by $\k_\C$ and $\t_\C$  the complexifications of $\k$ and $\t$, respectively. Then $\t_\C$ is a Cartan subalgebra of $\k_\C$, and we write  $\Sigma(\k_\C,\t_\C)$ for the corresponding system  of roots and $\Sigma^+$ for a set of positive roots. 
Now, as a consequence of the Cartan-Weyl classification of irreducible finite-dimensional representations of reductive Lie algebras over $\C$ one has the  identification 
\bqn
\widehat K\, \simeq\, \mklm{\Lambda \in \t_\C^\ast: \text{ $\Lambda$ is  dominant integral and $T$-integral}},
\eqn 
compare \cite{wallach}, 
and we write   $\Lambda_\sigma\in \t_\C^\ast$ for the \emph{highest weight} corresponding to  $\sigma \in \widehat K$ under  this isomorphism. Weyl's dimension formula then implies that $d_\sigma=O\big (|\Lambda_\sigma|  ^{|\Sigma^+|}\big )$, while from Weyl's character formula one infers that 
 if $D^u$  is a differential operator on $K$ of order $u$, 
\bq
 \label{eq:29.08.2015}
 \norm{D^u \sigma}_\infty = O\big (|\Lambda_\sigma|  ^{u+|\Sigma^+|}\big ), \qquad |\Lambda_\sigma|\to \infty,
\eq
compare \cite[Eq.\ (3.5)]{ramacher10add}. Consequently, the bound in Theorem \ref{thm:general1} can be rewritten as
\[
\|\phi_j\|_\infty \ll \sqrt{|\Lambda_\sigma|^{{2|\Sigma^+|}+\lfloor \frac{\dim K}2+1\rfloor}} \lambda_j^{\frac{\dim G/ K-1}{2m}-\delta}.
\]
\end{rem}

\begin{proof}[Proof of Theorem \ref{thm:general1}]
By translating  the results in \cite[Section 3]{Marshall2017} to our non-adelic setting, one verifies that the assumptions of Lemma \ref{lem:general} are fulfilled under the hypothesis  of the theorem. 
Note that it is unnecessary to relate the subgroup $K$ to the specific maximal connected compact subgroup considered in  \cite{Marshall2017}, because the assumptions in question are concerned only with the structure of the Hecke algebra and the lattice point counting function  $M(x,\alpha,\delta)$.

Let us explain this in a more detailed way. Since $H(\bA)=H(\Q)(H(\R) K_0)$ one has
\[
H(\Q)\bsl H(\bA) / K_0\cong \Gamma\bsl G.
\]
Now, any function $\varphi$ in $\L^2_\sigma(\Gamma\bsl G)$ can be  identified with a function $\varphi_\bA$ in $\L^2_\sigma(H(\Q)\bsl H(\bA) / K_0)$ by setting
\bq
\label{eq:identific}
\varphi_\bA(\gamma g k):= \varphi(g) , \qquad \gamma\in H(\Q) , \;\; g\in G , \;\; k\in K_0.
\eq
For each double coset $K_0\alpha K_0$ with $\alpha\in H(\bA_\mathrm{fin})$, a linear operator $T_{K_0\alpha K_0}$ on $\L^2_\sigma(H(\Q)\bsl H(\bA)/K_0)$ can then be  defined by setting
\[
(T_{K_0\alpha K_0} \phi_\bA)(x):= \sum_{h\in K_0\alpha K_0/K_0} \phi_\bA(xh),  \qquad \phi_\bA\in \L^2(H(\Q)\bsl H(\bA)/K_0).
\]
Moreover, there exist finitely many elements $\beta_1,\dots,\beta_m$ in $H(\Q)$ such that
\bq 
\label{eq:adeledisc}
H(\Q)\cap (H(\R)K_0\alpha^{-1} K_0)=\bigsqcup_{i=1}^m\Gamma \beta_i \Gamma,
\eq
the  intersection being non-empty due to the assumption $H(\bA)=H(\Q)(H(\R) K_0)$.\footnote{Without this assumption, the intersection in \eqref {eq:adeledisc}  might be empty, and the following arguments make no sense. Along the same lines, recall that  Hecke operators on $\mathrm{SO(3)}$ are only defined in the case $p\equiv 1\mod 4$. 
}
 This implies  that for all $\varphi\in \L^2_\sigma(\Gamma\bsl G)$
\bq
\label{eq:correspondence}
T_{K_0 \alpha K_0} \varphi_\bA=   \sum_{i=1}^m  T_{\Gamma\beta_i\Gamma}\varphi.
\eq
Hence, any adelic Hecke operator $T_{K_0\alpha K_0}$ can be regarded as a sum of non-adelic Hecke operators via the identification $\varphi\equiv\varphi_\bA$.

In order to apply  Lemma \ref{lem:general} in the present context, we choose $\chi=1$ and $\Xi=H(\Q)$.
Let $\mathcal{P}'$ be the set denoted by $\mathcal{P}$  in \cite[Section 2.5]{Marshall2017}, that is, an infinite subset of the totality of finite places of $F$.
A map $\mathfrak{N}':\mathcal{P}'\to\N$ is defined by the order of the residue field of $F_v$.
Then, $\uG(F_v)$ is split for each $v\in \mathcal{P}'$ and \eqref{eqnote1} holds by the prime ideal theorem and the Chebotarev density theorem. Now,  put
\bq
\label{eq:module}
\mathcal{H}:=\langle T_{K_0\alpha K_0} \mid   \alpha\in \uG(F_v), \; \; v\in\mathcal{P}'   \rangle.
\eq
Note that $\uG(F_v)\subset H(\Q_p)$ if $v|p$, and by \cite[Proof of Theorem 2.8.2 (2)]{Miyake} we have $(T_{K_0\alpha K_0})^*=T_{K_0\alpha^{-1} K_0}\in \mathcal{H}$.
Since $P_0$ commutes with all Hecke operators and each automorphic representation of $H(\bA)$ factors as a tensor product of irreducible unitary representations of $\uG(F_v)$ for all places $v$  \cite{Flath}, there exists an orthonormal basis $\{\phi_j\}_{j\geq 0}$ of $\L^2(\Gamma\bsl G)$ consisting of simultaneous eigenfunctions for $P$ and all $\T \in \mathcal{H}$. Now, let $\phi_{j_0} \in \L^2_\sigma(\Gamma \bsl G)$ be fixed. Applying the results in \cite{Marshall2017} to the function $\psi:=\phi_{j_0,\bA}$, that is also denoted by $\psi$ there, one can verify the assumptions of Lemma \ref{lem:general} for $\phi_{j_0}$. Indeed, by \cite[Propositions 6.1]{Marshall2017}, for each place $v\in \mathcal{P}'$ there exists a Hecke operator $\T_v\in\mathcal{H}$ such that $\T_v\psi=\psi$ holds and $\T_v$ is a linear combination of operators $T_{K_0\alpha K_0}$ with $\alpha\in \uG(F_v)$.
In view of \eqref{eq:correspondence} we can identify $\T_v$ with a non-adelic Hecke operator $\T_v'$ on $\L^2_\sigma(\Gamma\bsl G)$ such that 
\[
\T_v'\phi_{j_0} = \T_v\phi_{j_0,\bA} =\phi_{j_0,\bA} =\phi_{j_0}. 
\]
Similarly, set $\T_{N,x}:=\sum_{v\in\mathcal{Q}_{N,x}'} \T_v$, and denote the corresponding non-adelic Hecke operators by $\T_{N,x}'$, where $\mathcal{Q}_{N,x}'$ is chosen as the set denoted by $\mathcal{Q}_N$ in \cite[Section 3.4]{Marshall2017}. By the convolution on $H(\bA_\mathrm{fin})$, there exist $n\in\N$, $b_k\in\C$, and $\omega_k \in H(\bA_\mathrm{fin})$ such that
\[
\T_{N,x}\circ (\T_{N,x})^*=\sum_{k=1}^n b_k T_{K_0\omega_k K_0}.
\]
The corresponding  $\l \in \N$, $a_u\in\C$, and $\alpha_u \in C(\Gamma)$ in the decomposition of  $\T_{N,x}'\circ (\T_{N,x}')^*$ in Lemma \ref{lem:general} are then obtained from this equality via the identification \eqref{eq:correspondence}.
Finally, the upper bounds \eqref{eq:19.03.2017} in Lemma \ref{lem:general} can be verified using \eqref{eq:adeledisc}, \eqref{eq:correspondence} and the arguments in \cite[Section 3]{Marshall2017}, completing the proof of the theorem.
\end{proof}

\begin{ex}[\bf Equivariant subconvex bounds for $ \SL(n,\R)$]\label{ex:1}
Choose a central division algebra $D$ of index $n$ over $\Q$ such that $D\otimes \R\cong M(n,\R)$.
It is well known that the equivalence classes of central division algebras over $\Q$ are parameterized by the Brauer group $\mathrm{Br}(\Q)$, which can be realized as the set
\[
\Big\{(a,x) \mid a\in\{0,{1}/{2}\},\quad x=(x_p)\in\bigoplus_p \Q/\Z ,\quad a+\sum_p x_p=0 \mod \Z   \Big\},
\]
where $p$ runs over all primes, via the Brauer-Hasse-Noether theorem, compare \cite[Theorem 1.12]{PR}.
If we choose a prime $p_1$ and a parameter $(0,x)$ in $\mathrm{Br}(\Q)$ such that $x_{p_1}=a/n$ and $a$ is prime to $n$, then there is a central division algebra $D$ corresponding to $(0,x)$ and satisfying $D\otimes \R\cong M(n,\R)$.
A semisimple algebraic group $H$ over $\Q$ is then defined by $H:=\SL(1,D)$. Clearly,  $G:=H(\R)\cong \SL(n,\R)$, {$K:=\SO(n)$}, and $\Gamma:=H(\Q)\cap(H(\R)K_0) $ is cocompact for any open compact subgroup $K_0$ of $G(\mathbb{A}_\mathrm{fin})$, while  $H$ satisfies the condition (WS) and is simply connected.
Then, the assumptions of Theorem \ref{thm:general1} hold,  and for any Hecke--Maass form $\phi_j$  in $\L^2_\sigma(\Gamma\bsl G)$ with eigenvalue  $\lambda_j$ one obtains the subconvex bound
\[
\|\phi_j\|_\infty\ll \sqrt {d_\sigma \sup_{u \leq \lfloor \frac{n^2-n+4}4 \rfloor} \norm{D^u \sigma}_\infty} \, \,  \lambda_j^{\frac{n^2+n-4}{4m}-\delta}
\]
for some $\delta>0$.
\end{ex}

\begin{ex}[\bf Equivariant subconvex bounds for $\SL(n,\C)$] Let $D$ denote a central division algebra over a number field $F$.
Then $H:=\mathrm{Res}_{F/\Q}\SL(1,D)$ is a simply connected semisimple algebraic group over $\Q$.
The Brauer-Hasse-Noether theorem ensures that there exist various division algebras $D$ such that $H$ satisfies the assumptions of Theorem \ref{thm:general1}.
For example, by choosing suitable parameters in the Brauer group, one gets a central division algebra $D$ of index $n$ over an imaginary quadratic extension $F$ of $\Q$ such that $H(\R)$ is isomorphic to $\SL(n,\C)$.
Since $\SL(n,\C)$ satisfies the condition (WS), Theorem \ref{thm:general1} yields subconvex bounds for $\Gamma\bsl \SL(n,\C)$, where $\Gamma$ is defined by an open compact subgroup $K_0$ in $H(\bA_\mathrm{fin})$.
\end{ex}

\begin{ex}[\bf Equivariant subconvex bounds for $\SU(n,n,\R)$]
There exists a central division algebra $D$ over an imaginary quadratic extension $F$ of $\Q$ with a $F/\Q$-involution such that $H:=\SU(1,D)$ satisfies $G:=H(\R)\cong \SU(n,n,\R)$, see \cite[Theorem 8.1]{Scharlau}.
Since $H$ obviously satisfies the assumptions of Theorem \ref{thm:general1}, one obtains equivariant subconvex bounds for $\Gamma\bsl G$ in this case.
\end{ex}


\subsection{Non-equivariant subconvex bounds}
\label{sec:non-equiv general}

We shall now prove non-equivariant subconvex bounds on arithmetic quotients of semisimple algebraic groups without the condition (WS).
Let $H$ be a connected semisimple algebraic group over $\Q$ and choose any open compact subgroup $K_0$ of $H(\bA_\mathrm{fin})$.
We then have the following 
\begin{thm}
\label{thm:28.5.2017}
Put $G:=H(\R)$ and $\Gamma:=H(\Q)\cap(H(\R)K_0)$.
Assume that $H(\bA)=H(\Q)(H(\R)K_0)$ and that $\Gamma\bsl G$ is compact. Let $P_0$ be an elliptic left-invariant differential operator on $G$ of degree $m$ that gives rise to a  positive and symmetric operator on $\Gamma \bsl G$  with strictly convex cospheres $S^\ast_x(\Gamma \bsl G)$. Then, there exist a submodule $\mathcal{H}$ of $H^{\chi=1}_{\Xi=H(\Q)}$ and a constant $\delta>0$ such that there is an orthonormal basis $\{\phi_j\}_{j\in\N}$ of $\L^2(\Gamma\bsl G)$ consisting of simultaneous eigenfunctions for $P$ and all $\T \in \mathcal{H}$, so that  for each $\phi_j$ with spectral eigenvalue $\lambda_j$ one has
\[
\|\phi_j\|_\infty \ll \lambda_j^{\frac{\dim G-1}{2m}-\delta}.
\]
\end{thm}
\begin{proof}
{To prove this theorem, we need an explicit distance on $G$. We may assume that $H$ is a closed subgroup of $\SL(m)$ over $\Q$ for some sufficiently large $m \in \N$, so that $G=H(\R)$ becomes a closed subgroup of $\SL(m,\R)$ with respect to the topology induced from the Euclidean topology on $\R^{m^2}$, compare \cite[Chapter 3]{PR}.  Note that $H(\R)$ might consist of finitely many connected components with respect to the usual topology, even if $H$ is connected in the sense of  Zariski \cite[Corollary 1]{PR}. One then  defines on $M(m,\R)$ the  Euclidean distance
\bqn 
\dista(x,y):=\|x-y\|, \qquad \|x\|:=\mathrm{Tr}({}^t\!xx), \qquad x,y\in M(m,\R),
\eqn
obtaining a distance on $G$ by the inclusions $G\subset \SL(m,\R) \subset M(m,\R)$. In fact,  the distance $\dist$ is locally equivalent to the distance $\dista$. Indeed, $\dista$ is equivalent to $\dist$ in  a small neighborhood $\mathcal{U}$ of the identity. Furthermore, for fixed $g\in G$ one computes
\bqn
 \dista(gx,gy) \leq \|g\|\; \dista(x,y), \qquad \dista(gx,gy)\geq  \|g^{-1}\|^{-1}\; \dista(x,y),
 \eqn
so that  $\dista(x,y)$ is equivalent to the distance $(x,y)\mapsto\dista(gx,gy)$ on $G$. The assertion now follows by covering $G$ by  translates of $\mathcal{U}$.}

The first assertion follows from the corresponding argument  in Theorem \ref{thm:general1}. It remains to show that the assumptions in Lemma \ref{lem:general} are satisfied for the module $\mathcal{H}$ given in \eqref{eq:module}, for which we shall follow  the considerations in \cite{Marshall2017}.  Let us choose the same norm $\|\;\; \|^*$ as in \cite[Section 2.2]{Marshall2017} on the group of cocharacters of a maxial torus over $\overline{\Q}$, and regard  $\|\;\; \|^*$ as a norm on the cocharacters of each $\Q_p$-torus by conjugation.
Let $\mathcal{P}'$ be the set denoted by  $\mathcal{P}$ in \cite[Section 2.5]{Marshall2017} for $F=\Q$ and $\uG=H$. Then $\mathcal{P}'$ is an infinite  set of prime numbers, \eqref{eqnote1} holds, and for each prime $p\in \mathcal{P}'$ the group $H(\Q_p)$ is split.  Furthermore,  a Hecke operator $\tau(p,\mu)$ is defined by the product of $p^{-\|\mu\|^*}$ with the characteristic function of $H(\Z_p)\mu(p)H(\Z_p)$, where $\mu$ is a cocharacter on a suitable maximal split torus $\underline{T}_p$ in $H(\Q_p)$. In addition, several conditions are imposed on  $H$, $\mathcal{P}'$, and $\underline{T}_p$, and we refer the reader to \cite[Section 2]{Marshall2017} for  details. By \cite[Proposition 6.1]{Marshall2017}, there exists for each $p\in\mathcal{P}'$ a Hecke operator $\T_p$  such that
\[
\T_p\phi_{j,\bA}=\phi_{j,\bA},\quad \T_p=\sum_{\|\mu\|^*\leq R}a(p,\mu)\, \tau(p,\mu), \quad \T_p\T_p^*=\sum_{\|\mu\|^*\leq R}b(p,\mu)\,  \tau(p,\mu),
\]
\[
a(p,\mu)\ll 1 ,\quad a(p,0)=0 ,\quad b(p,\mu)\ll 1
\]
for some constant $R\in\N$.
Now, choose a compact subset $\Omega$ of $G$ such that $G=\Gamma\Omega$, let  $x$ be an element in $\Omega$, and set $\T_{N,x}:=\sum_{p\in\mathcal{P}',\, p\leq N} \T_p$.\footnote{Note that  $\T_{N,x} $ does actually not depend on $x$, but we preferred to keep the notation of Lemma \ref{lem:general}.} 
In order to verify the necessary conditions in  Lemma \ref{lem:general} we proceed  as in the proof of Theorem \ref{thm:general1}, and  let $\T_p'$ and $\T_{N,x}'$ be non-adelic Hecke operators corresponding to $\T_p$ and $\T_{N,x}$, respectively.
For $\gamma\in H(\Q)\subset \SL(m,\Q)$, let $\|\gamma\|_f$ denote the least common multiple of denominators of components of $\gamma$.
By \cite[Corollary 3.6]{Marshall2017}, one has $\|\gamma\|_f\ll N^{A_2'}$ for some $A_2'>0$ if $\gamma \in H(\Q)\cap\supp (\T_{N,x}(\T_{N,x})^*)$, where $\supp(\T_{N,x}(\T_{N,x})^*)$ means the support of $\T_{N,x}(\T_{N,x})^*$ in $H(\mathbb{A}_\mathrm{fin})$, and  the second bound in \eqref{eq:19.03.2017} follows.
{Furthermore, for the distance $\dist_1$, one can show that for some $A_2>A_2'$ the inequality $\mathrm{dist}_1(\gamma x,x)<c_1 N^{-A_2}$ does not hold for any non-trivial element $\gamma\in H(\Q)\cap \mathrm{supp}(\T_{N,x}(\T_{N,x})^*)$, where we choose $c_1>0$ such that $c_1 \dist(y_1,y_2)< \mathrm{dist}_1(y_1,y_2)$ for all $y_1,y_2\in \Gamma\bsl G$.}
Indeed, if $\gamma\in H(\Q)\cap \mathrm{supp}(\T_{N,x}(\T_{N,x})^*)$ is such that $\mathrm{dist}(\gamma x,x)< N^{-A_2}$,  one necessarily must have $\gamma=1$ together with  $\T_p(x^{-1}\gamma x)=\T_p(1)=0$.
It follows that
\[
\sum_{u=1}^l |a_u||\mathcal{M}(x,\alpha_u,N^{-A_2})| \leq \sum_{p\leq N} \sum_{\|\mu\|^*\leq R} | b(p,\mu)| \, \tau(p,\mu)(1)\ll N
\]
since
\[
\T_{N,x}(\T_{N,x})^*=\sum_{p\leq N} \sum_{\|\mu\|^*\leq R}  b(p,\mu) \, \tau(p,\mu) +\sum_{p\neq q \leq N}\sum_{\|\mu\|^*,\|\nu\|^*\leq R}a(p,\mu)\, a(q,\nu) \, \tau(p,\mu)\, \tau(q,-\nu),
\]
yielding the first bound in \eqref{eq:19.03.2017}. 
Thus, the proof is completed.
\end{proof}

\begin{ex}[\bf Non-equivariant subconvex bounds for division algebras]
Let $D$ be a central division algebra over a number field $F$ and $H:=\mathrm{Res}_{F/\Q}\SL(1,D)$.
If $H$ has no $\R$-simple components $H'$ for which  $H'(\R)$ is compact, then Theorem \ref{thm:28.5.2017} yields a subconvex bound for $\Gamma\bsl G$ with $G$ and $\Gamma$ as in the previous theorem. If   in particular $G=\SL(n,\R)$, there exists a constant $\delta>0$ such that one has the subconvex bound
\[
\|\phi_j\|_\infty\ll  \lambda_j^{\frac{n^2-2}{2m}-\delta}
\]
for any Hecke--Maass form $\phi_j$ in $\L^2(\Gamma\bsl G)$ with eigenvalue  $\lambda_j$. It is  also possible to consider $H:=\mathrm{Res}_{F/\Q}\SU(1,D)$, where $D$ is a central division algebra over a quadratic extension $E$ of $F$ equiped with a $E/F$-involution. Under the assumption that $G:=H(\R)$ has no $\R$-simple components $H'$ for which $H'(\R)$ is compact, Theorem \ref{thm:28.5.2017} then implies subconvex bounds for $\Gamma\bsl G$. Regarding the existence of  division algebras with $E/F$-involutions satisfying our assumptions we refer the reader to \cite[Theorem 8.1]{Scharlau}.
\end{ex}

\subsection{Automorphic forms and representation theoretic interpretation of the results}
\label{sec:autreprII} 
To close, let us indicate how our results fit into the general theory of automorphic forms \cite{borel-jacquet}. With $G$, $K$, and $\Gamma$ as above let us recall the following 

\begin{definition}
A smooth function $f:G\rightarrow \C$ satisfying
\begin{enumerate}
\item[(A1)] $f(\gamma g) =f(g) $ for all $g \in G$ and $\gamma \in \Gamma$,
\item[(A2)] $f$ is $K$-finite on the right,
\item[(A3)] $f$ is $\mathcal{Z}$-finite, where $\mathcal{Z}$ denotes the center of the universal envelopping algebra $\mathfrak{U}(\g_\C)$ of the complexification of the Lie algebra $\g$ of $G$,
\end{enumerate}
 is called an \emph{automorphic form on $G$ for $\Gamma$}.
\end{definition}

This definition  implies that  with respect to   the decomposition \eqref{eq:Gdecomp} each automorphic form $f$ is contained already in finitely many constituents,  so that   $f\in \bigoplus_{k=1}^l m(\pi_k,\Gamma)\cdot \pi_k $ for suitable $\pi_k$, compare \cite[Corollaries 8.14 and 10.37]{knapp}. Hence,  $p(\Ccal):=\prod_{k=1}^l(\Ccal-\mu_k)$, where $\mu_k$ is the Casimir eigenvalue of $\pi_k$, represents a polynomial that  annihilates $f$. On the other hand,  (A2) implies that  $f$ is a finite sum of functions $f_\sigma$ belonging to a  specific  $K$-type $\sigma$, and by \eqref{eq:casimir-laplace} one deduces 
\[
p(\Ccal)f_\sigma=\prod_{k=1}^l(2dR(\Omega_K)-\Delta-\mu_k)f_\sigma=\prod_{k=1}^l(2\mu_\sigma-\Delta-\mu_k)f_\sigma =:q(\Delta) f_\sigma
\]
where $\mu_\sigma$ denotes the eigenvalue of $dR(\Omega_K)$ on $\sigma$, and $q$ is a polynomial.  Thus, $f$ is essentially given by a sum of Hecke--Maass forms in the sense of this paper, because $q(\Delta)$ is an elliptic differential operator and any subspace defined by a $K$-type and a Casimir eigenvalue is finite dimensional by Harish-Chandra's theorem \cite[Theorem 1.7]{borel-jacquet}.
Consequently, Theorems \ref{thm:general1} and \ref{thm:28.5.2017} can be rephrased as follows.
\begin{thm}
\label{thm:26.9.2017}
Let $H$ be a connected semisimple algebraic group over $\Q$ and $K_0$ an open compact subgroup of $H(\bA_\mathrm{fin})$.
Set $G:=H(\R)$ and $\Gamma:=H(\Q)\cap (H(\R)K_0)$, and suppose that $H(\bA)=H(\Q)(H(\R)K_0)$  and $H(\Q)\bsl H(\bA)$ is compact.
Then, there exist a submodule $\mathcal{H}$ of $H^{\chi=1}_{\Xi=H(\Q)}$ and a constant $\delta>0$, which are independent of $\sigma\in\widehat{K}$, such that
\begin{enumerate}
\item there exists an orthonormal basis $\{\phi_j\}_{j\in\N}$ of $\L^2_\sigma(\Gamma\bsl G)$ which consists of simultaneous eigenfunctions for the Casimir operator $\Ccal$ and all Hecke operators $\T \in \mathcal{H}$;
\item for each $\phi_j$ with Casimir eigenvalue $\mu_j$ one has
\[
\|\phi_j\|_\infty \ll \sqrt {d_\sigma \sup_{u \leq \lfloor \frac{\dim K}2+1\rfloor} \norm{D^u \sigma}_\infty} \, \, (-\mu_j+2\mu_\sigma)^{\frac{\dim G/K-1}{4}-\delta},
\]
provided that  $H=\mathrm{Res}_{F/\Q}\underline{G}$ and {\rm (WS)} is fulfilled, while in general
\[
\|\phi_j\|_\infty \ll  (-\mu_j+2\mu_\sigma)^{\frac{\dim G-1}{4}-\delta}
\]
  $\mu_\sigma$ being the eigenvalue of $dR(\Omega_K)$ on $\sigma$. {If $K=T$ is a torus,
    \[
\|\phi_j\|_\infty \ll  \, (-\mu_j+2\mu_\sigma)^{\frac{\dim G/K-1}{4}-\delta}.
\]}

\end{enumerate}
\end{thm}
\qed

\providecommand{\bysame}{\leavevmode\hbox to3em{\hrulefill}\thinspace}
\providecommand{\MR}{\relax\ifhmode\unskip\space\fi MR }
\providecommand{\MRhref}[2]{%
  \href{http://www.ams.org/mathscinet-getitem?mr=#1}{#2}
}
\providecommand{\href}[2]{#2}




\end{document}